\documentclass{amsart}

\usepackage{amssymb,amsmath,cite}

\usepackage{amsfonts}
\usepackage{enumerate}
\usepackage{mathrsfs}
\usepackage{amscd}
\usepackage{amsmath}
\usepackage{latexsym}
\usepackage{amssymb}
\usepackage{amsthm}
\usepackage{graphicx}
\usepackage{color}
\usepackage{cite}
\usepackage{bm}

\newcommand{\al}{\alpha}
\newcommand{\be}{\beta}

\newcommand{\la}{\lambda}

\newcommand{\de}{\delta}
\newcommand{\si}{\sigma}

\newcommand{\ga}{\gamma}
\newcommand{\om}{\omega}
\newcommand{\Om}{\Omega}
\newcommand{\ka}{\kappa}

\makeatletter
\@addtoreset{equation}{section}

\newtheorem{proposition}{Proposition}[section]
\newtheorem{definition}{Definition}[section]
\newtheorem{lemma}{Lemma}[section]
\newtheorem{theorem}{Theorem}[section]

\theoremstyle{remark}
\newtheorem{remark}[theorem]{Remark}

\begin{document}

\title[]{Global pathwise solutions of an abstract stochastic equation}

\author{Y.-X. Lin}
\address{Yuan-Xin Lin
\newline\indent
School of Mathematical Sciences, Shanghai Jiao Tong University,
Shanghai, P. R. China}
\email{yuanxinlin@sjtu.edu.cn}

	\author{Y.-G. Wang}
\address{Ya-Guang Wang
	\newline\indent
	School of Mathematical Sciences, Center for Applied Mathematics, MOE-LSC and SHL-MAC, Shanghai Jiao Tong University,
	Shanghai, P. R. China}
\email{ygwang@sjtu.edu.cn  }
\maketitle

\date{}

\begin{abstract} We establish the existence and uniqueness of the maximal pathwise solution for an abstract nonlinear stochastic evolutional equation, which takes the two and three dimensional stochastic Navier-Stokes equations as a typical model, forced by a multiplicative white noise, and show that the pathwise solution exists globally in time in a positive probability when the initial data is sufficiently small. Moreover, a global pathwise solution is obtained for the stochastic Navier-Stokes equations defined on torus when the data is properly regular and small.
\end{abstract}
~

%
%
%
%
%

\section{Introduction}
In this paper, we consider the following initial value problem for an abstract nonlinear stochastic evolutional system which takes the stochastic Navier-Stokes equations as a special case:
\begin{equation}\label{ASE}
    \left\{
    \begin{aligned}
       du+[Au+B(u,u)]dt &=G(u)dW,\\
        u(0) &=u_0,
    \end{aligned}
\right.
\end{equation}
where $u$ is the unknown function, $A$ and $B$ are deterministic operators on $u$, and $W$ is a cylindrical Brownian motion evolving over a separable Hilbert space $\mathcal{H}$. For the detailed descriptions and hypotheses on these terms, see Section \ref{NH}.

When studying the theory of stochastic evolutional system, we usually consider two kinds of solutions. The first one is the martingale (probabilistic weak) solutions, whose underlying stochastic elements can change and are considered as parts of solutions. The second one is the pathwise (probabilistic strong) solutions, whose driving noise is given in advance. From the existence of martingale solutions and the uniqueness of pathwise solutions, we can obtain the existence of pathwise solutions by the classical Yamada–Watanabe lemma.

The well-posedness of the stochastic Navier-Stokes equations was first studied by Bensoussan and Temam in \cite{BT,BT2} in the 1970s. After half century's development, the research results in this field are abundant. If nonlinear multiplicative noise is applied, the problem can be considered in the framework of martingale solutions, see \cite{MV,ABC,MD,Flandoli,RB}, these works mainly consider the weak solution in the sense of PDE, that is the solution evolving in $L^2$. On the other hand, local and global pathwise solutions are considered in \cite{AJ,HB,RB,BP}. Afterwards, authors in \cite{G-H-Z} obtain the existence and uniqueness of the local pathwise solution of the initial boundary value problem when the initial data belongs to $H^1$, and the pathwise solution exists globally in two-dimensional case. The local existence of pathwise solutions for the initial value problem in $\mathbb{R}^3$ has been established in \cite{KIM}, and it proves that the global solution can be obtained under a certain positive probability as long as the initial data is sufficiently small and the driving noise is non-degenerate and multiplicative. Authors in \cite{De-G} obtain the local pathwise solutions to an abstract nonlinear stochastic evolutional system containing the Navier-Stokes equations as a special case, as a consequence, they also give another proof for the local existence result obtained in \cite{G-H-Z}.

The aim of this paper is to extend the work of \cite{De-G} to study the existence of the global pathwise solution for an abstract nonlinear evolutional system in the sense of positive probability. As a consequence, we obtain the existence of the global pathwise  solution for the initial boundary value problem of the stochastic Navier-Stokes equations, it is worth to note that the resulting probability is dependent on the viscosity. Furthermore, by improving the regularity of the initial data we get that the probability is independent of the viscosity. In addition, the hypotheses imposed on the noise in this paper are weaker than those in \cite{De-G}.

The exposition is organized as follows. In Section \ref{PM}, we give some hypotheses about the problem \eqref{ASE}, and state the main results. The existence and uniqueness of the maximal pathwise solution to the problem \eqref{ASE} under the hypotheses on the noise weaker than those in \cite{De-G} will be established in Section \ref{LMPS}. The global pathwise solution to the problem \eqref{ASE} will be obtained in Section \ref{GPS} under some further hypotheses. In Section \ref{NS}, we apply the results of the first two sections to the initial boundary value problem of the stochastic Navier-Stokes equations, and deduce the global pathwise solution under a certain positive probability depending on the viscosity. In Section \ref{NSM}, we study the stochastic Navier-Stokes equations defined on torus $\mathbb{T}^d$ (d=2 or 3) with the initial data being in $H^m$ ($m>\frac{d}{2}$+1) and small, and obtain  the global pathwise solution with a positive probability independent of the viscosity. In the Appendix, we collect several lemmas in stochastic analysis and compact embedding theorems used in this paper, such as the Burkholder-Davis-Gundy inequality, the Gronwall inequality for stochastic processes, etc.

\section{Preliminaries and main results}\label{PM}

To study the initial value problem \eqref{ASE} for an abstract nonlinear stochastic evolutional system, in this section we first give some notations and hypotheses on the operators $A$, $B(\cdot, \cdot)$ and the noise $G(u)$ appeared in the system \eqref{ASE}, then state the main results of this paper.
 
\subsection{Notations and Hypotheses}\label{NH}
We consider two separable Hilbert spaces $V \subset H$, with the embedding being dense and compact. Denote by $V'$ the dual of $V$, relative to $H$, and $(\cdot , \cdot)$, $((\cdot , \cdot ))$ and $|\cdot|$, $||\cdot||$ the inner products and norms in $H$, $V$ respectively. The duality pair between $V$ and $V'$ is denoted by $\langle \cdot , \cdot \rangle $.

\subsubsection{Hypotheses on operators $A$ and $B(\cdot, \cdot)$}\label{HA}
The hypotheses on the operator $A$ are the same as given in \cite{De-G}. Assume that $A: D(A)\subset H \rightarrow H $ is an unbounded, densely defined, bijective, linear operator such that
\begin{equation}
    ( Au,v ) =((u,v))\,\,\,\,\,\,\,\,\, \mathrm{for} \,\,\,\mathrm{all}\,\,\, u \in D(A),\,\,\,v \in  V,
\end{equation}
from which we know that the linear operator $A$ is symmetric and can be understood as a bounded operator from $V$ to $V'$ with the duality pair given by
\begin{equation}\label{AS}
    \langle Au,v \rangle =((u,v))\,\,\,\,\,\,\,\,\, \mathrm{for} \,\,\,\mathrm{all}\,\,\, u,v \in V.
\end{equation}
We further see that $A^{-1}$, the inverse of $A$, is continuous from $H$ into $V$, thus $A^{-1}$ is a compact operator on $H$, due to the embedding $V\subset H$ being compact. By applying the well known theory in functional analysis for the symmetric compact operator $A^{-1}$, we know that all the eigenfunctions of $A^{-1}$, denoted by $(e_j)_{j \geq 1}$, form an orthogonal base of $H$, and it is easy to verify that they are also eigenfunctions of $A$, that is
$$Ae_j=\la_je_j,\,\,\,\,\,\,\,\,\,e_j\in D(A),$$
and from the above assumptions, we know that the corresponding eigenvalues $(\la_j)_{j \geq 1}$ satisfying
$$ 0<\lambda_1 \leq \lambda_2 \leq \cdots \leq \lambda_n \leq \cdots \,\,\,\,\,\,\,\,\, \mathrm{and} \,\,\,\,\,\,\,\,\, \lim_{n\rightarrow \infty}\la_n=+\infty . $$

Introduce the finite dimensional spaces $H_n=span \{e_1,\cdots e_n\}$, which will be used in the Galerkin approximation later, and denote by $P_n$ and $Q_n=I - P_n$
the projections from $H$ to $H_n$ and the orthogonal complement of $H_n$ respectively.

For any given $\alpha >0$, we define
$$    D(A^\alpha) :=\left\{u\in H : \sum_{k=1}^{\infty}\lambda_k^{2\alpha}|(u,e_k)|^2< \infty\right\}$$
and the operator $A^\alpha : D(A^\alpha) \rightarrow H$ by

  $$  A^\alpha u := \sum_{k=1}^\infty \lambda_k^ \alpha (u,e_k) e_k,$$
equipped with the inner product and norm
   $$ \langle u, v \rangle _\alpha := \sum_{k=1}^\infty \lambda_k^{2\alpha}(u,e_k)(v,e_k),\,\,\,\,\,\,
    |u|_\alpha:=|A^\alpha u|=\left(\sum_{k=1}^\infty \lambda_k^{2\alpha}|(u,e_k)|^2\right)^{\frac{1}{2}}.$$
For any $0<\alpha_1 < \alpha_2$, it is easy to have the following inequalities:
\begin{equation}
    |P_nu|_{\alpha_2} \leq \lambda_n^{\alpha_2 - \alpha_1}|P_nu|_{\alpha_1},\,\,\,\,\,\,|Q_nu|_{\alpha_1} \leq \frac{1}{\lambda_n^{\alpha_2 - \alpha_1}}|Q_nu|_{\alpha_2}.
\end{equation}

Note that we have $D(A^{\frac{1}{2}})=V$ due to \eqref{AS}, and we can verify that the embedding $D(A^{\alpha_2}) \subset D(A^{\alpha_1})$ is compact for any $0< \alpha_1 < \alpha_2$ as in \cite{T}, therefore, we obtain the compact embedding $D(A) \subset V$.


As in \cite{De-G}, we impose the following assumptions on the bilinear term $B(\cdot , \cdot)$:
\begin{itemize}
    \item[$\bullet$] $B(u,v)$ is continuous from $V\times D(A) $ to $V'$, and satisfies
    \begin{equation}\label{B1}
    \langle B(u, v) , v \rangle =0 , \,\,\,\,\,\,\,\,\, \forall \,\,u\in V,\,\, v \in D(A),
\end{equation}
\begin{equation}\label{B2}
    |\langle B(u,v) , w \rangle| \leq C_B ||u||\,|Av|\, ||w||,\,\,\,\,\,\,\,\,\, \forall \,\,u,w \in V ,\,\, v\in D(A),
\end{equation}
for a constant $C_B>0$.
\item[$\bullet$]$B(u,v)$ is continuous from $D(A)\times D(A)$ to $H$, and satisfies
\begin{equation}\label{B3}
    |(B(u,v),w)|\leq C_B ||u||^{\frac{1}{2}} \, |Au|^{\frac{1}{2}} \, ||v||^{\frac{1}{2}} \, |Av|^{\frac{1}{2}} \, |w| , \,\,\,\,\,\,\,\,\, \forall u,v \in D(A),\,\,w \in H.
\end{equation}
\end{itemize}


In order to study the global pathwise solutions, we need to add a further assumption on $A$ and $B(\cdot , \cdot)$ as follows: 
    \begin{equation}\label{HGP}
    |(B(u,u),Au)|\leq C||u||^a\,|Au|^b, \,\,\,\,\,\,\,\,\, \forall u \in D(A),
\end{equation}
with two fixed constants $a,b$ satisfying $0<b<2$, $a+b>2$.

\subsubsection{Hypotheses on the noise}\label{AFN}
  In this subsection, we introduce the assumptions on the stochastic driving term $G(u)dW$ given in \eqref{ASE}.

  We first introduce the concept of stochastic basis. For a given probability space $(\Om,\mathcal{F},\mathbb{P})$ equipped with a complete, right-continuous filtration $\{\mathcal{F}_t\}_{t\geq 0}$, which means that $\mathbb{P}(A)=0$ implies $A\in \mathcal{F}_0$ and $\mathcal{F}_t=\bigcap_{s>t}\mathcal{F}_s$, and $\{\beta_k\}_{k\ge 1}$ a sequence of mutually independent standard real valued  Brownian motions  relative to $\mathcal{F}_t$, we say that the quintuple
  $$\mathcal{S}=\left(\Om,\mathcal{F},\{\mathcal{F}_t\}_{t\geq 0},\mathbb{P},\{\beta_k\}_{k\geq 1}\right),$$
  is a stochastic basis.

  Fix a separable Hilbert space $\mathcal{H}$ with an orthonormal basis $\{h_k\}_{k\geq 1}$, we say that $W$ is a cylindrical Brownian motion evolving over $\mathcal{H}$ if it has the following formal expansion:
  \begin{equation}\label{FE}
      W=\sum_{k=1}^\infty \beta_k h_k.
  \end{equation}
  Fix another separable Hilbert space $\mathcal{U}$, we denote the collection of Hilbert-Schmidt operators between $\mathcal{H}$ and $\mathcal{U}$ by $L_2(\mathcal{H},\mathcal{U})$, which contains all bounded linear operators $A\in L(\mathcal{H},\mathcal{U})$ such that
  $$||A||_{L_2(\mathcal{H},\mathcal{U})}^2:=\sum_{k=1}^\infty||Ah_k||_\mathcal{U}^2< \infty.$$
  For an $\mathcal{U}$-valued predictable process $G \in L^2(\Om;L^2_{loc}([0,\infty),L_2(\mathcal{H},\mathcal{U})))$, one may define the Itô stochastic integral
  \begin{equation}
      M_t:=\int_0^tGdW=\sum_{k=1}^\infty \int_0^t Gh_k \, d\beta_k,
  \end{equation}
  see \cite{GJ} for detail construction.

  One needs to note that the expansion \eqref{FE} does not converge in general in $\mathcal{H}$, to have the convergent, as in \cite{GJ}, we need to introduce a large space $\mathcal{H}_0 \supset \mathcal{H}$ defined by
  $$\mathcal{H}_0:= \left\{ h=\sum_{k\geq 1} a_kh_k : \sum_{k\geq 1} \frac{a_k^2}{k^2}<\infty  \right\} ,$$
  endowed with the usual norm $||h||_{\mathcal{H}_0}:= \left( \sum_{k\geq 1} \frac{a_k^2}{k^2} \right)^{\frac{1}{2}}$. One can obtain that $W \in C([0,\infty); \mathcal{H}_0)$ almost surely.

  Now let's introduce some hypotheses on $G$, first introduce the following notations. For given Banach spaces $\mathcal{X},\mathcal{Y},\mathcal{Z}$ with $\mathcal{X} \subset \mathcal{Z}$, we define
  \begin{equation}
      \begin{aligned}
      Bnd_{u,loc}(\mathcal{X},\mathcal{Y};\mathcal{Z}) := & \{f \in C([0,\infty)\times \mathcal{X}) :  \\
      & ||f(t,x)||_{\mathcal{Y}} \leq  \beta (||x||_{\mathcal{Z}}) (1+||x||_{\mathcal{X}}) , \,\,\,
  \forall x \in \mathcal{X}, t \geq 0\},
      \end{aligned}
  \end{equation}
  where $\beta(\cdot)$ is a positive increasing function being locally bounded and independent of $t$. In addition, we define the space of locally Lipschitz functions by
  \begin{equation}
      \begin{aligned}
       Lip_{u,loc}(\mathcal{X},\mathcal{Y};\mathcal{Z}) := &  \{f \in Bnd_{u,loc}(\mathcal{X},\mathcal{Y};\mathcal{Z}) : \\
 & ||f(t,x)-f(t,y)||_{\mathcal{Y}} \leq  \beta (||x||_{\mathcal{Z}}+||y||_{\mathcal{Z}}) ||x-y||_{\mathcal{X}} , \,\,\forall x,y \in \mathcal{X}, t \geq 0\}.
      \end{aligned}
  \end{equation}
  For the noise term given in the problem \eqref{ASE}, we assume that $G:[0,\infty) \times H \rightarrow L_2(\mathcal{H};H)$ satisfies
    \begin{equation}\label{G-Bnd}
        G\in Bnd_{u,loc}(H,L_2(\mathcal{H},H);H) \cap Bnd_{u,loc}(V,L_2(\mathcal{H},V);H)  \cap Bnd_{u,loc}(D(A),L_2(\mathcal{H},D(A));H),
    \end{equation}
  and
   \begin{equation}\label{G-Lip}
        G\in Lip_{u,loc}(H,L_2(\mathcal{H},H);H) \cap Lip_{u,loc}(V,L_2(\mathcal{H},V);H) \cap Lip_{u,loc}(D(A),L_2(\mathcal{H},D(A));H),
    \end{equation}
  where the condition \eqref{G-Bnd} is used in the proof of existence of martingale solutions and the condition \eqref{G-Lip} is used in the proof of existence of pathwise solutions. Note that the hypotheses imposed on the noise term in \cite{De-G} are the global Lipschitz conditions, but we assume that the noise term satisfies the local Lipschitz conditions only in this paper.


In considering the global pathwise solutions, we further assume that $G(u)$ satisfies
  \begin{equation}\label{G-H1}
    ||((G(u),u))||_{L_2{(\mathcal{H},\mathbb{R})}} ^2 \geq \beta ^2 ||u||^4,
\end{equation}
\begin{equation}\label{G-H2}
    ||G(u)||_{L_2{(\mathcal{H},V)}} ^2 \leq \alpha^2 ||u||^2,
\end{equation}
for two constants $\alpha , \beta$ satisfying
\begin{equation}\label{G-H3}
    \alpha^2 < 2 \beta^2.
\end{equation}

\subsection{Main results}
Before stating the main results, as in \cite{De-G,G-H-Z}, we introduce several definitions for solutions to the problem \eqref{ASE}. The first one is the martingale solution, which is weak in the probability sense, and its stochastic elements may change and be a part of solutions.
\begin{definition}(Local Martingale Solution)
Suppose that $\mu_0$ is a probability measure on $V$. A triple $(\Tilde{\mathcal{S}},\Tilde{u},\Tilde{\tau})$ is a local Martingale solution to the problem \eqref{ASE}, provided that
$\Tilde{\mathcal{S}}=(\Tilde{\Omega},\Tilde{\mathcal{F}},\{ \Tilde{\mathcal{F}_t} \}_{t \geq 0},\Tilde{\mathbb{P}},\Tilde{W})$ is a stochastic basis, $\Tilde{\tau}$ is a $\Tilde{\mathcal{F}_t}$-stopping time and $\Tilde{u}(\cdot)=\Tilde{u}(\cdot
\land \Tilde{\tau}): \Tilde{\Omega} \times [0,\infty)\rightarrow V$ is $\Tilde{\mathcal{F}_t}$-adapted such that

\begin{equation}\label{E1}
\begin{aligned}
&  \Tilde{u}(\cdot \land \Tilde{\tau})\in L^2(\Tilde{\Om};C([0,\infty);V)), \\
&  \Tilde{u}\mathbb{I}_{t\leq \Tilde{\tau}} \in L^2(\Tilde{\Om};L^2_{loc}([0,\infty);D(A))),
\end{aligned}
\end{equation}
the law of $\Tilde{u}_0$ is $\mu_0$, and for any $t \geq 0$ there holds in $H$ that
\begin{equation}\label{E2}
    \Tilde{u}(t\land \Tilde{\tau})+\int_0^{t\land \Tilde{\tau}}(A\Tilde{u}+B(\Tilde{u},\Tilde{u}))ds=\Tilde{u}(0)+\int_0^{t\land \Tilde{\tau}}G(\Tilde{u})d\Tilde{W}.
\end{equation}

\end{definition}

The second one is the pathwise solution, which is strong in the probability sense, and its probability space is the original one.
\begin{definition}\label{Def2}(Local Pathwise Solution)
Fix a stochastic basis $\mathcal{S}=(\Omega,\mathcal{F}, \{ \mathcal{F}_t \}_{t \geq 0},\mathbb{P},W)$ and assume that $u_0$ is a $V$-valued, $\mathcal{F}_0$-measurable random variable. A pair $(u,\tau)$ is called a local pathwise solution to the problem \eqref{ASE} provided that $\tau$ is a strictly positive stopping time and $u(\cdot \land \tau)$ is a $\mathcal{F}_t$-adapted process in $V$ such that \eqref{E1} and \eqref{E2} holds without tildes.
\end{definition}

\begin{definition}(Maximal Pathwise Solution)\label{MS}
We say that a triple $(u,\{ \tau_n \}_{n\in \mathbb{N}},\tau^M)$ is a maximal pathwise solution to the problem \eqref{ASE} if the following holds:
\begin{itemize}
    \item[$\bullet$] $\tau^M$ is a $\mathcal{F}_t$-stopping time and strictly positive a.s.,
    \item[$\bullet$] The pair $(u,\tau_n)$ is a local pathwise solution for each n,
    \item[$\bullet$]  $\{ \tau_n \}_{n\in \mathbb{N}}$ is an increasing sequence of $\mathcal{F}_t$-stopping times such that $\tau_n \rightarrow \tau^M$ a.s., and
\begin{equation}\label{BU}
    \sup_{t\in [0,\tau^M]}||u||^2+\int_0^{\tau^M} |Au|^2 ds = \infty,
\end{equation}
almost surely on $\{ \tau^M < \infty \}$. \\
\end{itemize}
In addition, if
\begin{equation}\label{FBU}
    \sup_{t\in [0,\tau_n]}||u||^2+\int_0^{\tau_n} |Au|^2 ds = n,
\end{equation}
is true for almost surely $\om \in \{\tau^M < \infty \} $, then we say $\{ \tau_n \}_{n\in \mathbb{N}}$ announces any finite time blow up.
\end{definition}

\begin{definition}(Global Pathwise Solution)\label{DGPS}
In Definition \ref{MS}, if there exists $\epsilon \in (0,1)$, such that 
\begin{equation}\label{Global Sense}
    \mathbb{P}(\tau^M=\infty)>1-\epsilon,
\end{equation}
then we call that $(u, \tau^M)$ is a global pathwise solution
in a positive probability sense.
\end{definition}

Now, we state the main results of this paper. The first one concerns the existence of a unique maximal pathwise solution to the problem \eqref{ASE}, which will be established in Section \ref{LMPS}.

\begin{theorem}\label{MPS}
Let $\mathcal{S}=(\Om, \mathcal{F}, (\mathcal{F}_t)_{t \geq 0}, \mathbb{P}, W)$ be a fixed stochastic basis. For the problem \eqref{ASE}, assume that the stochastic force satisfies the hypotheses \eqref{G-Bnd}-\eqref{G-Lip}, and the initial data $u_0$ is a $V$-valued, $\mathcal{F}_0$-measurable random variable, satisfying
\begin{equation}
    u_0\in L^2(\Om,V).
\end{equation}
Then there exists a unique maximal pathwise solution $(u,\{\tau_n\}_{n \in \mathbb{N}}, \tau^M)$ to the problem \eqref{ASE}. Moreover, we can construct $\{\tau_n\}_{n \in \mathbb{N}}$ to make it announces any finite time blow up.

\begin{remark}
As noted here, for simplicity, we omit the deterministic force  $F(u)$ in the problem \eqref{ASE}, the conclusion of Theorem \ref{MPS} can be extended to the situation with a deterministic external force, with certain hypotheses imposed on $F(u)$ as given in \cite{De-G}.
\end{remark}

\end{theorem}

The second result is the existence of a global pathwise solution to the problem \eqref{ASE}, which will be obtained in Section \ref{GPS}.
\begin{theorem}\label{GPT}
Under the hypotheses in Theorem \ref{MPS}, we further assume that $G(u)$ satisfies \eqref{G-H1}-\eqref{G-H3}, and \eqref{HGP} holds for the deterministic terms $A$ and $B$. Then for any $\epsilon \in (0,1)$, there exists some $\delta =\delta(\epsilon) >0$, such that for any initial data $u_0$ satisfying $\mathbb{E}||u_0||\leq \delta$, the pathwise solution constructed in Theorem \ref{MPS} is global in the sense of Definition \ref{DGPS}.

\end{theorem}

Applying Theorems \ref{MPS} and \ref{GPT} to the initial boundary value problem of the stochastic Navier-Stokes equations \eqref{SNS}, we immediately come to the following conclusion, which shall be discussed in detail in Section \ref{NS}.
\begin{theorem}\label{GPSSNS}
Assume that the hypotheses for noise and initial data in Theorem \ref{GPT} hold for the problem \eqref{SNS}. Then there exists a  unique maximal pathwise solution $(u,\{\tau_n\}_{n \in \mathbb{N}}, \tau^M)$ to the problem \eqref{SNS}. Furthermore, for any $\epsilon \in (0,1)$, there exists some $\delta =\delta(\epsilon,\nu) >0$, such that for any initial data $u_0$ satisfying $\mathbb{E}||u_0||\leq \delta$, the pathwise solution is global in the sense of Definition \ref{DGPS}.

\end{theorem}

Note that the $\delta$ in Theorem \ref{GPSSNS} depends on the viscosity $\nu$, in Section \ref{NSM} we shall show that if we consider the stochastic Navier-Stokes equations defined on the torus $\mathbb{T}^2$ or $\mathbb{T}^3$, then the $\delta$ can be independent of $\nu$ providing that the initial data is sufficiently smooth and small, namely in $H^m$ ($m > \frac{d}{2}+1 $). Denote by $X^m$ the divergence free subspace of $H^m$.

\begin{theorem}\label{GPSmT}
For a fixed $m > \frac{d}{2}+1$, let $\mathcal{S}=(\Om, \mathcal{F}, \{ \mathcal{F}_t \}_{t \geq 0}, \mathbb{P}, W)$ be a fixed stochastic basis, for the problem \eqref{SNS} defined on the torus $\mathbb{T}^2$ or $\mathbb{T}^3$, assume that the stochastic force satisfies the hypotheses \eqref{GPSm-H1}-\eqref{GPSm-H4} and \eqref{G-H3}, and the initial data $u_0$ is a $X^m$-valued, $\mathcal{F}_0$-measurable random variable. Then for any $ \epsilon \in (0,1)$, there exists $\de=\de(\epsilon) > 0 $ such that for any initial data $u_0$ satisfying $\mathbb{E}||u_0||_{H^m} \leq \de$, the problem \eqref{SNS} admits a global pathwise solution in $C([0,\infty);X^m)$.

\end{theorem}

\section{Local existence of maximal pathwise solutions}\label{LMPS}

In this section, we prove Theorem \ref{MPS}. Let's first establish the following proposition.
\begin{proposition}\label{WI}
Let $\mathcal{S}=(\Om, \mathcal{F}, (\mathcal{F}_t)_{t \geq 0}, \mathbb{P}, W)$ be a fixed stochastic basis, for the problem \eqref{ASE}, assume that the stochastic force satisfies the hypotheses \eqref{G-Bnd}-\eqref{G-Lip}, and the initial data $u_0$ is a $V$-valued, $\mathcal{F}_0$-measurable random variable, satisfying
\begin{equation}\label{WH}
    ||u_0|| \leq M , \,\,\,\,\,\,\,\,\, \mathrm{a.s.}\,\,\om \in \Om,
\end{equation}
for a positive constant M. Then there exists a unique local pathwise solution $(u,\tau)$ to the problem \eqref{ASE}.

\end{proposition}

Inspired by \cite{De-G}, noting that $G$ satisfies the hypotheses \eqref{G-Bnd}-\eqref{G-Lip}, we introduce the following cut-off problem:
\begin{equation}\label{CS}
    \left\{
    \begin{aligned}
       du+(Au+\theta (||u-u_*||)B(u,u))dt &=\theta (||u-u_*||)G(u)dW,\\
        u(0) &=u_0,
    \end{aligned}
\right.
\end{equation}
where $\theta :\mathbb{R} \rightarrow [0,1] $ is a smooth cut-off function satisfying
\begin{equation}\label{CF}
    \theta(x)=
\begin{cases}
1, &|x| \leq \kappa,\\
0, & |x|\geq 2\kappa,
\end{cases}
\end{equation}
for some constants $\ka \leq \frac{1}{64C_B}$, with $C_B$ being the constant given in \eqref{B2}, and $u_*$ is the solution to the following linear problem:
\begin{equation}\label{LS}
    \left\{
    \begin{aligned}
       \frac{d}{dt}u_*+Au_* &=0,\\
        u_*(0) &=u_0.
    \end{aligned}
\right.
\end{equation}

Similar to \cite{De-G}, we use the following Galerkin scheme to  solve the problem \eqref{CS},
\begin{equation}\label{GS}
    \left\{
    \begin{aligned}
       du^n+(Au^n+\theta(||u^n-u^n_*||)B^n(u^n,u^n))dt &= \theta(||u^n-u^n_*||)G^n(u^n)dW,\\
        u^n(0) &= P_n u_0, 
    \end{aligned}
\right.
\end{equation}
where $B^n(u,u)=P_nB(u,u),\,\,\, G^n(u)=P_nG(u),\,\,\, u^n_*=P_nu_*$, with $P_n$ being the projection from $H$ to $H_n=span \{e_1,\cdots e_n\}$. In the following, we use the notation $B(u)$ instead of $B(u,u)$ for convenience. For each $n\in \mathbb{N}$, the existence of a unique continuous adapted solution $\{u^n(t)\}_{t\geq 0}$, taking value in $H_n$, to the problem \eqref{GS} is standard thanks to the  cancelation property \eqref{B1}, we refer to \cite{Flandoli} for detail.

By projecting the linear problem \eqref{LS} into the finite dimensional space $H_n$, we obtain
\begin{equation}\label{LS2}
    \left\{
    \begin{aligned}
       \frac{d}{dt}u_*^n+Au_*^n &=0,\\
        u_*^n(0) =u_0^n:&=P_nu_0.
    \end{aligned}
\right.
\end{equation}
As in \cite{De-G}, we come to the following conclusion.

\begin{lemma}
For each $p \geq 2$ and $n \in \mathbb{N}$, the solution $u_*^n$ to the problem \eqref{LS2} satisfies the following two estimates:
\begin{equation}\label{EFLS}
    \sup_{t\in[0,T]}||u_*^n||^p + \int_0^T |Au_*^n|^2 ||u_*^n||^{p-2} dt + \left(\int_0^T |Au_*^n|^2 dt\right) ^{\frac{p}{2}} \leq C||u_0||^p,
\end{equation}
and
\begin{equation}\label{EFLS2}
    ||u_*^n||_{W^{1,2}(0,T;H)} \leq C \int_0^T |Au_*^n|^2 dt \leq C||u_0||^2.
\end{equation}
\end{lemma}

To prove Proposition \ref{WI}, we first establish some uniform estimates for the approximation solutions $\{u^n\}_{n \in \mathbb{N}}$, then obtain the existence of local martingale solutions to the problem \eqref{ASE} by using the stochastic compactness method. Secondly, the so-called 'Pathwise Uniqueness' proposition is proved, and finally, the existence of pathwise solutions to the problem \eqref{ASE} will be obtained by using the Gyongy-Krylov theory, Lemma \ref{Gyongy-Krylov}.

\subsection{Uniform estimates}
\begin{lemma}\label{UEs}
Under the hypotheses given in Proposition \ref{WI}, for a fixed $p \geq 2$ and any $T>0$, there exists a constant $C=C(p,\ka,||u_0||,T)>0$, such that the approximate solutions $\{u^n\}_{n \in \mathbb{N}}$ constructed in \eqref{GS} satisfying 
\begin{equation}\label{UE1}
    \sup_{n \in \mathbb{N}} \mathbb{E}\left(\sup_{t\in[0,T]}||u^n||^p+\int_0^T|Au^n|^2||u^n||^{p-2}dt\right) \leq C,
\end{equation}
\begin{equation}\label{UE2}
    \sup_{n \in \mathbb{N}} \mathbb{E}\left(\int_0^T|Au^n|^2dt\right)^{\frac{p}{2}}\leq C,
\end{equation}

\begin{equation}\label{UE3}
    \sup_{n \in \mathbb{N}} \mathbb{E}\left|\left|\int_0^t \theta(||u^n-u^n_*||)G^n(u^n) dW\right|\right|^p_{W^{\alpha,p}([0,T];H)} \leq C,
\end{equation}
where $\al \in (0,\frac{1}{2})$, and
\begin{equation}\label{UE4}
    \sup_{n \in \mathbb{N}} \mathbb{E}\left|\left|u^n(t)-\int_0^t \theta(||u^n-u^n_*||)G^n(u^n) dW\right|\right|^2_{W^{1,2}([0,T];H)} \leq C.
\end{equation}
\end{lemma}

\begin{proof}

The proof of this lemma is similar to Lemma 3.1 in \cite{De-G}, here we only deal with terms containing $G$ and refer to \cite{De-G} for the remaining terms. The constant $C$ in the following may change from line to line.

  Define $\Bar{u}^n=u^n-u^n_*$, then $\Bar{u}^n$ satisfies
$$
      \left\{
    \begin{aligned}
       d\Bar{u}^n+[A\Bar{u}^n+\theta(||\Bar{u}^n||)B^n(\Bar{u}^n+u^n_*)]dt &=\theta(||\Bar{u}^n||)G^n(\Bar{u}^n+u^n_*)dW,\\
        \Bar{u}^n(0) &=0.
    \end{aligned}
\right.
$$
With the Itô formula for $||\Bar{u}^n||^p$, we infer that for $p \geq 2$,
\begin{equation}\label{ue1}
\begin{aligned}
 d||\Bar{u}^n||^p+p|A\Bar{u}^n|^2||\Bar{u}^n||^{p-2} dt
            &=\frac{p}{2} \theta^2(||\Bar{u}^n||)||G^n(\Bar{u}^n+u^n_*)||_{L_2(\mathcal{H},V)}^2||\Bar{u}^n||^{p-2} dt\\
            &\,\,\,\,\,\,+\frac{p(p-2)}{2}\langle \theta(||\Bar{u}^n||)G^n(\Bar{u}^n+u^n_*),A\Bar{u}^n\rangle ^2||\Bar{u}^n||^{p-4}dt\\
            &\,\,\,\,\,\,-p\theta(||\Bar{u}^n||)\langle B^n(\Bar{u}^n+u^n_*),A\Bar{u}^n\rangle ||\Bar{u}^n||^{p-2}dt\\
            &\,\,\,\,\,\,+p||\Bar{u}^n||^{p-2}\langle \theta(||\Bar{u}^n||) G^n(\Bar{u}^n+u^n_*),A\Bar{u}^n\rangle dW\\
            &=(J_1^p+J_2^p+J_3^p)dt+J_4^pdW.
\end{aligned}
\end{equation}
The hypothesis \eqref{G-Bnd} and Young's inequality yield
\begin{equation}\label{ue2}
\begin{aligned}
 |J_1^p| & \leq C \, \theta^2(||\Bar{u}^n||)\beta^2(|\Bar{u}^n+u^n_*|)(1+||\Bar{u}^n+u^n_*||)^2||\Bar{u}^n||^{p-2}\\
 & \leq C \left((1+||u^n_*||)^2+||\Bar{u}^n||^2\right)||\Bar{u}^n||^{p-2} \\
 & \leq C \left((1+||u^n_*||)^p+||\Bar{u}^n||^p\right).
\end{aligned}
\end{equation}
Using \eqref{AS} we know
\begin{equation}\label{ue3}
\begin{aligned}
 |J_2^p| & \leq C \theta^2(||\Bar{u}^n||)||G^n(\Bar{u}^n+u^n_*)||_{L_2(\mathcal{H},V)}^2||\Bar{u}^n||^{p-2} \\
 & \leq C \left((1+||u^n_*||)^p+||\Bar{u}^n||^p\right).
\end{aligned}
\end{equation}
The term $J_3^p$ was treated in Lemma 3.1 in \cite{De-G}, they obtained that
\begin{equation}\label{ue8}
    |J_3^p| \leq \frac{p}{2} ||\Bar{u}^n||^{p-2}|A\Bar{u}^n|^2 + C||u^n_*||^2|Au^n_*|^2.
\end{equation}
Finally, we deal with the term $J_4^p$. For any pair of stopping times $\tau_a$ and $\tau_b$, satisfying $0 \leq \tau_a \leq \tau_b \leq T$, by using Lemma \ref{BDG} and the hypothesis \eqref{G-Bnd}, we have
 \begin{equation}\label{ue4}
\begin{aligned}
  & \mathbb{E} \sup_{\tau_a \leq t \leq \tau_b}\left|\int_{\tau_a}^tJ_4^pdW\right| \\
            &\leq C\mathbb{E}\left(\int_{\tau_a}^{\tau_b}p^2||\Bar{u}^n||^{2(p-2)}\theta^2(||\Bar{u}^n||)\left|\left|\left\langle G^n(\Bar{u}^n+u^n_*),A\Bar{u}^n\right\rangle\right|\right|_{L_2(\mathcal{H},\mathbb{R})}^2 ds\right)^{\frac{1}{2}}\\
            &\leq C \mathbb{E}\left(\int_{\tau_a}^{\tau_b}||\Bar{u}^n||^{2(p-1)}\theta^2(||\Bar{u}^n||)\beta^2(||\Bar{u}^n||+||u^n_*||)(1+||\Bar{u}^n||+||u^n_*||)^2ds\right)^{\frac{1}{2}}\\
            &  \leq C \mathbb{E}\sup_{\tau_a \leq t \leq \tau_b}||\Bar{u}^n||^{p-1} \left(\int_{\tau_a}^{\tau_b}(1+||\Bar{u}^n||+||u^n_*||)^2ds\right)^{\frac{1}{2}}\\
            &\leq \frac{1}{2} \mathbb{E} \sup_{\tau_a\leq t\leq \tau_b}||\Bar{u}^n||^p+C \mathbb{E}\left(\int_{\tau_a}^{\tau_b}(1+||\Bar{u}^n||+||u^n_*||)^2ds\right)^{\frac{p}{2}}\\
            &\leq \frac{1}{2} \mathbb{E} \sup_{\tau_a\leq t\leq \tau_b}||\Bar{u}^n||^p+C\mathbb{E}\int_{\tau_a}^{\tau_b}\left((1+||u^n_*||)^p+||\Bar{u}^n||^p\right)ds.
\end{aligned}
\end{equation}
Integrating \eqref{ue1} from $\tau_a$ to $t$, taking supremum in $t \in [\tau_a , \tau_b]$, taking expectation, and using \eqref{ue2}-\eqref{ue4}, we deduce
\begin{equation}\label{ue5}
    \begin{aligned}
    & \mathbb{E}\left(\sup_{[\tau_a,\tau_b]} ||\Bar{u}^n||^p + \int_{\tau_a}^{\tau_b} |A\Bar{u}^n|^2||\Bar{u}^n||^{p-2} dt\right) \\
    & \leq C \mathbb{E}\left(||\Bar{u}^n(\tau_a)|| + \int_{\tau_a}^{\tau_b} \left(||\Bar{u}^n||^p+1+|Au^n_*|^2||u^n_*||^2+||u^n_*||^p\right)ds
    \right).
    \end{aligned}
\end{equation}
Therefore, by using Lemma \ref{SGL} and \eqref{EFLS}, for any $t \in (0,T]$, it holds that
\begin{equation}\label{ue6}
    \begin{aligned}
    & \mathbb{E}\left(\sup_{[0,t]} ||\Bar{u}^n||^p + \int_0^t |A\Bar{u}^n|^2||\Bar{u}^n||^{p-2} ds\right) \\
    & \leq C\mathbb{E}\left(\int_0^t \left(1+|Au^n_*|^2||u^n_*||^2+||u^n_*||^p\right)ds\right) \\
    & \leq C.
    \end{aligned}
\end{equation}
From \eqref{ue6} to conclude \eqref{UE1} and \eqref{UE2} is similar to Lemma 3.1 in \cite{De-G}.


By using Lemma \ref{SL} and the hypothesis \eqref{G-Bnd}, we conclude that
\begin{equation}\label{ue7}
\begin{aligned}
    & \mathbb{E}\left|\left|\int_0^t \theta(||u^n-u^n_*||)G^n(u^n) dW\right|\right|^p_{W^{\alpha,p}([0,T];H)} \\
    & \leq C\mathbb{E}\left(\int_0^T ||\theta(||u^n-u^n_*||)G^n(u^n)||^p_{L_2(\mathcal{H},H)}\right) \\
    & \leq C \mathbb{E}\left(\int_0^T \beta^p(2\kappa+||u_*^n||)\left(1+|u^n|^p\right)ds\right) \\
    & \leq C\mathbb{E}\int_0^T\left(1+|u^n|^p\right)ds,
   \end{aligned}
\end{equation}
from which we conclude \eqref{UE3} by using \eqref{UE1}. As for \eqref{UE4} we refer the readers to Lemma 3.1 in \cite{De-G} for detail.
\end{proof}
\subsection{Local existence of martingale solutions}
Define the phase spaces
  $$  \chi_u := L^2(0,T;V) \cap  C([0,T];V'),\,\,\,\,\,\,\,\,\,
    \chi_W := C([0,T];\mathcal{H}_0),\,\,\,\,\,\,\,\,\,
    \chi := \chi_u \times \chi_W. $$

Denote by $\mathcal{D}(u^n)$, $\mathcal{D}(W)$ the laws of $u^n$, $W$ respectively, which are the probability measures on $\chi_u$, $\chi_W$ respectively. More precisely,
\begin{equation}
    \mathcal{D}(u^n)(E)=\mathbb{P}(u^n\in E) \,\,\,\,\,\,\,\,\,\mathrm{for \,\,any} \,\,\,E\subset  \chi_u,
\end{equation}
\begin{equation}
    \mathcal{D}(W)(E)=\mathbb{P}(W\in E)\,\,\,\,\,\,\,\,\,\mathrm{for \,\,any} \,\,\,E\subset  \chi_W,
\end{equation}
thus $\mu^n:=\mathcal{D}(u^n)\times \mathcal{D}(W)$ defines a probability measure on the phase space $\chi$.

Let us denote by $\mu_0$ the law of $u_0$ on $V$, the assumption \eqref{WH} on the initial data $u_0$ implies
\begin{equation}\label{AFIL}
    \int_V ||u||^q d\mu_0(u)<\infty,
\end{equation}
for any positive $q$. Thus the uniform estimates established in Lemma \ref{UEs} show that $\{\mu^n\}_{n \in \mathbb{N}}$ is tight, and by the Prokhorov theorem, Lemma \ref{Prokhorov}, it is weakly compact in $\chi$.
\begin{lemma}\label{WC}
The sequence $\{\mu^n\}_{n \in \mathbb{N}}$ is tight
and weakly compact in the phase space $\chi$.

\end{lemma}
\begin{proof}
Applying the compact embedding conclusions in Lemma \ref{EL}, combining with the uniform estimates established in Lemma \ref{UEs}, we can verify that $\{\mu^n\}_{n \in \mathbb{N}}$ is tight, and on this basis we obtain that it is weakly compact over the phase space $\chi$ by directly using Lemma \ref{Prokhorov}. See Lemma 4.1 in \cite{De-G} for further detail.
\end{proof}

Lemma \ref{WC} shows that the sequence $\{\mu^n\}_{n \in \mathbb{N}}$ is weakly compact in $\chi$, so there exists a subsequence $\{\mu^{n_k}\}_{k \in \mathbb{N}}$ weakly convergent in $\chi$, thus using the Skorokhod theorem, Lemma \ref{Skorokhod}, directly on $\{\mu^{n_k}\}_{k \in \mathbb{N}}$ we obtain a new probability space and a new sequence of random variables, which has the laws $\{\mu^{n_k}\}_{k \in \mathbb{N}}$, and converges almost surely in the new probability space, we still denote the above subsequence $\{\mu^{n_k}\}_{k \in \mathbb{N}}$ by $\{\mu^n\}_{n \in \mathbb{N}}$ for convenience.
\begin{proposition}\label{Sk}
 Under the hypotheses in Proposition \ref{WI}, there exists a probability space $(\Tilde{\Omega},\Tilde{\mathcal{F}},\Tilde{\mathbb{P}})$ and a sequence of $\chi$-valued random variables  $\{(\Tilde{u}^n,\Tilde{W}^n)\}_{n \in \mathbb{N}}$, which has the law $\{\mu^n\}_{n \in \mathbb{N}}$, such that
$\{(\Tilde{u}^n,\Tilde{W}^n)\}_{n \in \mathbb{N}}$ converges almost surely, in the topology of $\chi$,  to an element $(\Tilde{u},\Tilde{W})$, i.e.
\begin{equation}\label{Sk1}
    \begin{aligned}
    & \Tilde{u}^n \rightarrow \Tilde{u}\,\,\,\,\,\,\,\,\,\mathrm{in}\,\,L^2(0,T;V) \cap  C([0,T];V'), \\
    & \Tilde{W}^n \rightarrow \Tilde{W}\,\,\,\,\,\,\,\,\,\mathrm{in}\,\,C([0,T];\mathcal{H}_0).
    \end{aligned}
\end{equation}
\end{proposition}

We observe that each pair $(\Tilde{u}^n,\Tilde{W}^n)$ is a solution to
\begin{equation}\label{MGE}
    d\Tilde{u}^n+[A\Tilde{u}^n+\theta(||\Tilde{u}^n-\Tilde{u}^n_*||)B^n(\Tilde{u}^n)]dt=\theta(||\Tilde{u}^n-\Tilde{u}^n_*||)G^n(\Tilde{u}^n)d\Tilde{W}^n,
    \end{equation}
with the initial data
\begin{equation}
        \Tilde{u}^n(0)=\Tilde{u}^n_0 := P_n\Tilde{u}^n(0),
\end{equation}
where $\Tilde{u}^n_*$ satisfies
\begin{equation}
\left\{
\begin{aligned}
     d\Tilde{u}^n_*+A\Tilde{u}^n_*dt &=0, \\
     \Tilde{u}^n_*(0) &=\Tilde{u}^n_0.
\end{aligned}
\right.
\end{equation}
Meanwhile, $\Tilde{W}^n$ is a cylindrical Wiener process relative to $\Tilde{\mathcal{F}}_t^n$, which is the completion of $\sigma(\Tilde{W}^n(s),\Tilde{u}^n(s):s \leq t)$. Techniques similar to those in \cite{Bensoussan} lead to the above conclusions.

By passing to the limit of each term in \eqref{MGE} and introducing an appropriate stopping time we can obtain the existence of local martingale solutions to the problem \eqref{ASE}.

\begin{proposition}\label{PTL}
  Let $\Tilde{\mathcal{S}}=(\Tilde{\Omega},\Tilde{\mathcal{F}},\{\Tilde{\mathcal{F}}_t\}_{t \geq 0},\Tilde{\mathbb{P}},\Tilde{W})$, where $\Tilde{\mathcal{F}}_t$ is the completion of $\sigma(\Tilde{u}(s),\Tilde{W}(s); s \leq t)$, define the stopping time
  \begin{equation}
   \Tilde{\tau} :=\inf_{t \geq 0}\{||\Tilde{u}-\Tilde{u}_*||\geq \kappa\}.
  \end{equation}
  Then $(\Tilde{\mathcal{S}},\Tilde{u},\Tilde{\tau})$ is a local martingale solution to the problem \eqref{ASE}.
\end{proposition}
\begin{proof}
  We only show the convergence of the stochastic term, while the convergence of the remaining terms and the verification of the solution have been obtained in Section 7 of \cite{De-G}. Obviously, one has
  \begin{equation}\label{MS3}
\begin{aligned}
     & ||\theta(||\Tilde{u}^n-\Tilde{u}^n_*||)G^n(\Tilde{u}^n)-\theta(||\Tilde{u}-\Tilde{u}_*||)G(\Tilde{u})||_{L_2(\mathcal{H},H)}\\
            &\leq\theta(||\Tilde{u}^n-\Tilde{u}^n_*||)||G^n(\Tilde{u}^n)-G(\Tilde{u})||_{L_2(\mathcal{H},H)} +[\theta(||\Tilde{u}^n-\Tilde{u}^n_*||)-\theta(||\Tilde{u}-\Tilde{u}_*||)]||G(\Tilde{u})||_{L_2(\mathcal{H},H)}.
   \end{aligned}
\end{equation}
First we claim the following estimate which will be established below:
\begin{equation}\label{MS2}
    \theta(||\Tilde{u}^n-\Tilde{u}^n_*||)||G^n(\Tilde{u}^n)-G(\Tilde{u})||_{L_2(\mathcal{H},H)} \rightarrow 0 ,\,\,\,\,\,\, \mathrm{for}\,\,\, \mathrm{a.e.}\,\,\,\, \mathrm{(\omega,t)} \in \Tilde{\Omega}\times [0,T].
\end{equation}
Note that
$$   ||G(\Tilde{u})||_{L_2(\mathcal{H},H)}\leq \Tilde{M} < \infty, $$
since $||\Tilde{u}^n-\Tilde{u}^n_*||\rightarrow ||\Tilde{u}-\Tilde{u}_*|| $ and the smoothness of $\theta$, we have
\begin{equation}
    \theta(||\Tilde{u}^n-\Tilde{u}^n_*||) \rightarrow \theta(||\Tilde{u}-\Tilde{u}_*||)\,\,\,\,\,\,\,\mathrm{for}\,\,\,\mathrm{a.e.}\,\,\,\,\mathrm{(\omega,t)} \in \Tilde{\Omega}\times [0,T].
\end{equation}
Therefore, from \eqref{MS3} we conclude that
\begin{equation}\label{MS1}
 ||\theta(||\Tilde{u}^n-\Tilde{u}^n_*||)G^n(\Tilde{u}^n)-\theta(||\Tilde{u}-\Tilde{u}_*||)G(\Tilde{u})||_{L_2(\mathcal{H},H)}\rightarrow 0\,\,\,\,\,\,\mathrm{for}\,\,\,\mathrm{a.e.}\,\,\,\,\mathrm{(\omega,t)} \in \Tilde{\Omega}\times [0,T].
\end{equation}

On the other hand,
 \begin{equation}
\begin{aligned}
      & \sup_{n}\mathbb{E}\left(\int_0^T ||\theta(||\Tilde{u}^n-\Tilde{u}^n_*||)G^n(\Tilde{u}^n)||_{L_2(\mathcal{H},H)}^2 ds\right)\\
            &\leq\sup_{n}\mathbb{E}\left(\int_0^T \theta^2 (||\Tilde{u}^n-\Tilde{u}^n_*||) \beta^2 (|\Tilde{u}^n|)(1+||\Tilde{u}^n||)^2 ds\right) \\
            &\leq C \sup_n \mathbb{E}\int_0^T(1+||\Tilde{u}^n||)^2 ds,
   \end{aligned}
\end{equation}
which means $|| \theta(||\Tilde{u}^n-\Tilde{u}^n_*||)G^n(\Tilde{u}^n) ||_ {L_2(\mathcal{H},H)}$ is uniformly integrable in $L^2 (\Tilde{\Omega}\times [0,T])$. Combining with \eqref{MS1} and using the Vitali convergence theorem, we obtain
\begin{equation}\label{CLP}
   \theta(||\Tilde{u}^n-\Tilde{u}^n_*||)G^n(\Tilde{u}^n) \rightarrow \theta(||\Tilde{u}-\Tilde{u}_*||)G(\Tilde{u}) \,\,\,\,\,\,\mathrm{in}\,\,\,L^2(\Tilde{\Omega};L^2([0,T];L_2(\mathcal{H},H))).
\end{equation}
Moreover, we conclude that the convergence in \eqref{CLP} holds in probability in $L^2([0,T];L_2(\mathcal{H},H))$.
Since $\Tilde{W}^n\rightarrow \Tilde{W}$ in $C([0,T];\mathcal{H}_0)$ in probability (c.f. Proposition \ref{Sk}),
by applying Lemma \ref{SCL} we deduce that
 \begin{equation}
\begin{aligned}
     \int_0^t \theta(||\Tilde{u}^n-\Tilde{u}^n_*||)G^n(\Tilde{u}^n)d\Tilde{W}^n \rightarrow \int_0^t \theta(||\Tilde{u}-\Tilde{u}_*||)G^n(\Tilde{u})d\Tilde{W}\,\,\,\,\,\,&  \mathrm{in}\,\, L^2([0,T];H)\,\,\\
     & \mathrm{in}\,\,\mathrm{probability}.
   \end{aligned}
\end{equation}

Now, let us verify \eqref{MS2}:
 $$\begin{aligned}
             &  \theta \left(||\Tilde{u}^n-\Tilde{u}^n_*||\right)||G^n(\Tilde{u}^n)-G(\Tilde{u})||_{L_2(\mathcal{H},H)}\\
            &\leq  \theta\left(||\Tilde{u}^n-\Tilde{u}^n_*||\right)||G(\Tilde{u}^n)-G(\Tilde{u})||_{L_2(\mathcal{H},H)} +\theta\left(||\Tilde{u}^n-\Tilde{u}^n_*||\right)||Q_nG(\Tilde{u})||_{L_2(\mathcal{H},H)}\\
            &\leq \theta\left(||\Tilde{u}^n-\Tilde{u}^n_*||\right)\beta\left(|\Tilde{u}^n|+|\Tilde{u}|\right)\left|\left| \Tilde{u}^n-\Tilde{u} \right| \right|+\theta\left(||\Tilde{u}^n-\Tilde{u}^n_*||\right)\frac{1}{\lambda_n^{\frac{1}{2}}}||G(\Tilde{u})||_{L_2(\mathcal{H},V)}\\
            &\leq C\left(||\Tilde{u}^n-\Tilde{u}||+\frac{1}{\lambda_n^{\frac{1}{2}}}(1+||\Tilde{u}||)\right),
        \end{aligned}$$
thus \eqref{MS2} holds for almost every $\mathrm{(\Tilde{\omega},t)} \in \Tilde{\Omega} \times  [0,T]$.
\end{proof}
\subsection{Local existence of pathwise solutions}
Now we consider the local pathwise solutions of the problem \eqref{ASE}, we first establish the following 'Pathwise Uniqueness' result.
\begin{proposition}\label{LPS}
  Let $\tau$ be an almost surely positive stopping time. Suppose that $(\Tilde{\mathcal{S}},u^{(1)},\tau)$ and $(\Tilde{\mathcal{S}},u^{(2)}, \tau)$ are two local martingale solutions to the problem \eqref{ASE} relative to the same stochastic basis.
  Define  $\Omega_0=\{u^{(1)}(0)=u^{(2)}(0)\}$, then $u^{(1)}$ and $u^{(2)}$ are indistinguishable in the sense that
  \begin{equation}\label{PU7}
      \mathbb{P}\left\{\mathbb{I}_{\Omega_0}\left(u^{(1)}(t\land \tau)-u^{(2)}(t\land \tau ) \right) =0; \,\,\,\forall t \geq 0 \right\} = 1.
  \end{equation}
\end{proposition}
\begin{proof}
  Define $w=u^{(1)}-u^{(2)}$, then the system for $w$ reads
\begin{equation}
\begin{aligned}
dw &+\left[Aw+B(u^{(1)})-B(u^{(2)})\right]dt =\left[G(u^{(1)})
-G(u^{(2)})\right]dW,
\end{aligned}
\end{equation}
with the initial data
\begin{equation}
    w(0)=u^{(1)}(0)-u^{(2)}(0).
\end{equation}
The Itô formula for $||w||^2$ yields
\begin{equation}\label{PU1}
\begin{aligned}
d||w||^2+2|Aw|^2 &= 2\left \langle B(u^{(2)})-B(u^{(1)}),Aw\right\rangle dt+\left|\left|G(u^{(1)})-G(u^{(2)})\right|\right|_{L^2(\mathcal{H},V)}^2 dt \\
& \,\,\,\,\,\,+ 2 \left\langle G(u^{(1)})-G(u^{(2)}),Aw \right \rangle dW.
\end{aligned}
\end{equation}


For $K>0$, introducing the stopping time
\begin{equation}\label{STSK}
    \xi^{(K)}:=\inf_{t > 0} \left \{\sup_{s \in [0,t]} \left(||u^{(1)}||^2+||u^{(2)}||^2 \right) +\int_0^t \left( ||u^{(1)}||^2|Au^{(1)}|^2+||u^{(2)}||^2|Au^{(2)}|^2 \right) ds \geq K \right \}.
\end{equation}

Fix $K>0$ and $T>0$, for any stopping times $\tau_a$ and $\tau_b$ such that $0\leq \tau_a \leq \tau_b \leq \tau \land \xi^{(K)} \land T$. Integrating \eqref{PU1} in time, taking supremum, multiplying by $\mathbb{I}_{\Omega_0}$, and taking the expectation, we obtain
\begin{equation}\label{PU2}
\begin{aligned}
& \mathbb{E}  \mathbb{I}_{\Omega_0}\left(\sup_{s \in [\tau_a,\tau_b]}||w||^2+\int_{\tau_a}^{\tau_b}|Aw|^2 ds\right) \\
&=\mathbb{E} \mathbb{I}_{\Omega_0} ||w(\tau_a)||^2+2\mathbb{E}\mathbb{I}_{\Omega_0}\int_{\tau_a}^{\tau_b} \left |\left \langle B(u^{(2)})-B(u^{(1)}),Aw \right \rangle \right| ds \\
&\,\,\,\,\,\,+\mathbb{E}\mathbb{I}_{\Omega_0}\int_{\tau_a}^{\tau_b}  \left|\left|G(u^{(1)}) - G(u^{(2)})\right|\right|_{L^2(\mathcal{H},V)}^2 ds\\
&\,\,\,\,\,\,+2\mathbb{E}\mathbb{I}_{\Omega_0}\sup_{s \in [\tau_a,\tau_b]}\left|\int_{\tau_a}^s \left\langle
\left(G(u^{(1)}) -  G(u^{(2)} ) \right) , Aw \right\rangle dW\right|\\
&:=\mathbb{E} \mathbb{I}_{\Omega_0} ||w(\tau_a)||^2+J_1+J_2+J_3.
\end{aligned}
\end{equation}
The hypothesis \eqref{G-Lip} implies
\begin{equation}\label{PU3}
\begin{aligned}
J_2
& \leq  C \mathbb{E} \mathbb{I}_{\Omega_0} \int_{\tau_a}^{\tau_b} \beta^2 \left(|u^{(1)}|+|u^{(2)}|\right)\,||u^{(1)}-u^{(2)}||^2 \, ds\\
&\leq C(K) \mathbb{E} \mathbb{I}_{\Omega_0} \int_{\tau_a}^{\tau_b} ||w||^2 \, ds.
\end{aligned}
\end{equation}
By using Lemma \ref{BDG}, the hypothesis \eqref{G-Lip} and Young's inequality, we find
\begin{equation}\label{PU8}
    \begin{aligned}
    J_3 & \leq C \mathbb{E}\mathbb{I}_{\Omega_0} \left(\int_{\tau_a}^{\tau_b} \left|\left|G(u^{(1)}) - G(u^{(2)})\right|\right|_{L^2(\mathcal{H},V)}^2  ||w||^2 ds\right)^{\frac{1}{2}}  \\
    & \leq C \mathbb{E}\mathbb{I}_{\Omega_0} \left(\sup_{s \in [\tau_a,\tau_b]} ||w||^2 \int_{\tau_a}^{\tau_b} \left|\left|G(u^{(1)}) - G(u^{(2)})\right|\right|_{L^2(\mathcal{H},V)}^2 ds \right)^{\frac{1}{2}} \\
    & \leq \frac{1}{2} \mathbb{E}\mathbb{I}_{\Omega_0} \sup_{s \in [\tau_a,\tau_b]} ||w||^2 + C(K)\, \mathbb{E}\mathbb{I}_{\Omega_0}\int_{\tau_a}^{\tau_b} ||w||^2 ds.
    \end{aligned}
\end{equation}
The term $J_1$ was estimated by using the bilinearity of $B$ and the hypothesis \eqref{B3}, see \cite[Proposition 5.1]{De-G}, they obtained that
\begin{equation}\label{PU9}
    \begin{aligned}
    J_1  & \leq  \frac{1}{4} \mathbb{E} \mathbb{I}_{\Omega_0} \int_{\tau_a}^{\tau_b} |Aw|^2 ds + C \mathbb{E} \mathbb{I}_{\Omega_0} \int_{\tau_a}^{\tau_b} \left(||u^{(1)}||^2|Au^{(1)}|^2 + ||u^{(2)}||^2|Au^{(2)}|^2\right) ||w||^2 \, ds  \\
    & \leq  \frac{1}{4} \mathbb{E} \mathbb{I}_{\Omega_0} \int_{\tau_a}^{\tau_b} |Aw|^2 ds + C(K) \mathbb{E} \mathbb{I}_{\Omega_0} \int_{\tau_a}^{\tau_b} ||w||^2 \, ds.
    \end{aligned}
\end{equation}
Combining \eqref{PU2}-\eqref{PU9}, we deduce
\begin{equation}\label{PU6}
\begin{aligned}
 \mathbb{E}\left(\sup_{t \in [\tau_a,\tau_b]}\mathbb{I}_{\Omega_0}||w||^2+\int_{\tau_a}^{\tau_b}\mathbb{I}_{\Omega_0}|Aw|^2 ds\right)\leq C \mathbb{E}\mathbb{I}_{\Omega_0}||w(\tau_a)||^2 +C(K)\mathbb{E}\int_{\tau_a}^{\tau_b} \mathbb{I}_{\Omega_0}||w||^2 ds,
\end{aligned}
\end{equation}
where C and C(K) are independent of $\tau_a$ and $\tau_b$, thus from Lemma \ref{SGL}, we conclude that
\begin{equation}
    \mathbb{E}\sup_{t \in [0,\tau \land \xi^{(K)} \land T]}\mathbb{I}_{\Omega_0}||w||^2 + \int_0^{\tau \land \xi^{(K)} \land T}\mathbb{I}_{\Omega_0}|Aw|^2 ds \leq C(K,T)\mathbb{E}\mathbb{I}_{\Omega_0}||w(0)||^2=0,
\end{equation}
which implies
\begin{equation}\label{L2}
    \mathbb{E}\sup_{t \in [0,\tau \land \xi^{(K)} \land T]}\mathbb{I}_{\Omega_0}||w||^2 =0,
\end{equation}
in particular, we have
\begin{equation}\label{M}
   \mathbb{I}_{\Omega_0}||w(\tau \land \xi^{(K)} \land T)||^2 =0,\,\,\,\,\,\,\mathrm{a.s.}
\end{equation}
Note that
\begin{equation}\label{L3}
    \begin{aligned}
    \mathbb{P} \left( \{ \xi^{(K)} < \tau \} \right)  \leq \mathbb{P} \left(  \{ \xi^{(K)} < \tau \} \cap \left \{ \tau < \infty \right \} \right) + \mathbb{P} \left( \{ \xi^{(K)} < \tau \} \cap \{ \tau = \infty \} \right). 
    \end{aligned}
\end{equation}
As in \cite[Propositon 5.1]{De-G}, we know that $\lim_{K\rightarrow \infty} \xi^{(K)} = \infty $ as $\tau = \infty$, hence the second term on the right hand side of \eqref{L3} vanishes. For the first term, the Chebyshev's inequality yields
\begin{equation}\label{L4}
    \begin{aligned}
    & \mathbb{P} \left(\{ \xi^{(K)} < \tau \} \cap \{ \tau < \infty \} \right) \\ 
    & \leq \mathbb{P} \left ( \left \{ \sup_{s \in [0,\tau]} \left( ||u^{(1)}||^2+||u^{(2)}||^2 \right) +\int_0^ \tau \left( ||u^{(1)}||^2|Au^{(1)}|^2+||u^{(2)}||^2|Au^{(2)}|^2 \right) ds \geq K  \right \} \cap \{ \tau < \infty \} \right) \\ 
    & \leq C \left ( \frac{1}{K} +\frac{1}{K^{\frac{1}{2}}} \right ) \mathbb{E} \left (\sup_{s \in [0,\tau]} \left (||u^{(1)}||^2+||u^{(2)}||^2 \right ) \right ) + \frac{C}{K^{\frac{1}{2}}} \mathbb{E} \mathbb{I}_{\{ \tau < \infty \}} \left( \int_0^ \tau \left( |Au^{(1)}|^2+|Au^{(2)}|^2 \right) ds \right),
    \end{aligned}
\end{equation}
the right hand side of \eqref{L4} tends to $0$ as $K \rightarrow \infty $ since \eqref{E1}, so dose the left hand side of \eqref{L3}. Therefore, for any $T>0$, from \eqref{M}-\eqref{L4}, we conclude that
\begin{equation}
   \mathbb{P} \left( \mathbb{I}_{\Omega_0}||w(\tau \land T)||^2 =0 \right) = 1.
\end{equation}
Following the same arguments as in \cite[Proposition 4.1]{G-H-Z}, we finally conclude \eqref{PU7}, the proof is complete.
\end{proof}

With Propositions \ref{PTL} and \ref{LPS} in hand, we can use Lemma \ref{Gyongy-Krylov} to obtain the existence of local pathwise solutions for the problem \eqref{ASE}, we refer to \cite{De-G} or \cite{G-H} for further detail of the procedure. Thus we complete the proof of Proposition \ref{WI}.

Via a localization argument we can extend Proposition \ref{WI} to a class of general initial data $u_0\in L^2(\Om,V)$. Indeed, for any $k \in \mathbb{N}$, the initial data
$$u_{k,0}=u_0 \mathbb{I}_{k \leq ||u_0|| < k+1},$$
satisfying
$$||u_{k,0}|| \leq k+1\,\,\,\,\,\,\,\,\,\mathrm{a.s.},$$
thus from Proposition \ref{WI} we can obtain the local pathwise solutions $(u_k,\tau_k)$ to the problem \eqref{ASE} with respect to the initial data $u_{k,0}$. Let
$$u=\sum_{k=0}^\infty u_k \mathbb{I}_{k \leq ||u_0|| < k+1},\,\,\,\,\,\,\,\,\,\mathrm{and}\,\,\,\,\,\,\,\,\, \tau=\sum_{k=0}^\infty\tau_k\mathbb{I}_{k \leq ||u_0||<k+1}.$$
One can show that $(u,\tau)$ is a local pathwise solution to the problem \eqref{ASE} with respect to the initial data $u_0$. We refer to \cite{G-H-Z} for specific processing procedure. To sum up the above discussion, we can draw the following conclusion.
\begin{proposition}\label{GI}
For a fixed stochastic basis $\mathcal{S}=(\Om, \mathcal{F}, (\mathcal{F}_t)_{t \geq 0}, \mathbb{P}, W)$. Assume that $u_0$ is a $V$-valued random variable such that $u_0$ is $\mathcal{F}_0$-measurable with $u_0\in L^2(\Om,V)$ and \eqref{G-Bnd}-\eqref{G-Lip} hold for $G$. Then there exists a unique local pathwise solution $(u, \tau)$ to the problem \eqref{ASE}.
\end{proposition}

Finally, with Proposition \ref{GI} in hand, we can go from the local pathwise solutions to the maximal pathwise solutions, we refer to \cite{Jacob} or \cite{G-H-Z2} for detail. The proof of Theorem \ref{MPS} is now completed.

\section{Global existence of pathwise solutions in positive probability}\label{GPS}
Let $(u,\{\tau_n\}_{n \in \mathbb{N}},\tau^M)$ be the maximal pathwise solution to the problem \eqref{ASE}. The goal of this section is to show that the maximal pathwise solution is global in time in positive probability when the initial data is properly small. Inspired by \cite{KIM}, we need to utilize the following key lemma, whose proof is similar to Lemma 5.1 in \cite{KIM}, which we shall give in Appendix \ref{PKL} for completeness.
\begin{lemma}\label{GP}
 For any $\xi>0$, let
$$\si_\xi := \inf  \{t>0: ||u(t)||\geq \xi \},$$
then $\si_\xi<\tau^M$ on the set $\{\tau^M<\infty\}$. In other words, for almost surely $\omega \in \Omega $, either $\tau^M = \infty $ or $\si_\xi < \tau^M.$
\end{lemma}

\noindent \textbf{Proof of Theorem \ref{GPT}:}
For any $\eta>0$ and $\la \in (0,1)$, the Itô formula for $(||u||^2+\eta)^{\frac{\la}{2}}$ yields
\begin{equation}\label{G1}
\begin{aligned}
 d\left(||u||^2+\eta\right)^{\frac{\la}{2}}  &= \la\left(||u||^2+\eta\right)^\frac{\la-2}{2}\left\{-|Au|^2-(B(u),Au)+\frac{1}{2}||G(u)||_{L_2(\mathcal{H},V)}\right\}dt \\
 &\,\,\,\,\,\,+ \frac{\la(\la-2)}{2}\left(||u||^2+\eta\right)^{\frac{\la-4}{2}} ||((G(u),u))_{L_2(\mathcal{H},\mathbb{R})}||^2 \, dt \\
 &\,\,\,\,\,\,+\la \left(||u||^2+\eta\right)^{\frac{\la-2}{2}}((G(u),u)) \, dW \\
 &=(I_1+I_2) \, dt + I_3 \, dW.
\end{aligned}
\end{equation}
Using the assumption \eqref{HGP} and Young's inequality, we obtain
$$|(B(u),Au)|\leq C||u||^a|Au|^b\leq C_1||u||^{2+r}+\frac{1}{2}|Au|^2,$$
where $r=\frac{2(a+b-2)}{2-b}>0$, combining with the assumption \eqref{G-H2} we find
\begin{equation}\label{G2}
    \begin{aligned}
    I_1 &\leq \la \left(||u||^2+\eta\right)^{\frac{\la-2}{2}}\left\{-|Au|^2+C_1||u||^{2+r}+\frac{1}{2}|Au|^2+\frac{\al^2}{2}||u||^2\right\} \\
    & \leq \la \left(||u||^2+\eta\right)^{\frac{\la-2}{2}} ||u||^2 \left\{C_1||u||^r+\frac{\al^2}{2}\right\}.
    \end{aligned}
\end{equation}
Using $\la \in (0,1)$ and the hypothesis \eqref{G-H1} we control $I_2$ as the following:
\begin{equation}\label{G3}
    \begin{aligned}
    I_2 &\leq \frac{\la(\la-2)}{2}\left(||u||^2+\eta\right)^{\frac{\la-4}{2}}\be^2||u||^4 \\
    &=\frac{\la(\la-2)}{2}\left(||u||^2+\eta\right)^{\frac{\la-2}{2}}\be^2||u||^2-\frac{\la(\la-2)}{2}\left(||u||^2+\eta\right)^{\frac{\la-4}{2}}\be^2||u||^2\eta  \\
    &\leq \frac{\la(\la-2)}{2}(||u||^2+\eta)^{\frac{\la-2}{2}}\be^2||u||^2-\frac{\la(\la-2)}{2}\be^2 \eta^{\frac{\la}{2}}.
    \end{aligned}
\end{equation}
The term $I_3$ can be handled by using Lemma \ref{BDG}, more specifically, for any stopping time $\si$ such that $\si<\tau^M$ on the set $\{\tau^M<\infty \}$, we have
\begin{equation}\label{G4}
    \begin{aligned}
    & \mathbb{E}\sup_{s\in [0,t\land \si]}\left|\int_0^sI_3\,dW\right| \\
    &\leq C \mathbb{E}\left(\int_0^{t\land \si}\la^2\left(||u||^2+\eta\right)^{\la-2}\left|\left|((G(u),u))\right|\right|_{L_2(\mathcal{H},\mathbb{R})}^2ds\right)^{\frac{1}{2}} \\
    &\leq C\mathbb{E}\left(\sup_{s\in [0,t\land \si]}\left(||u||^2+\eta\right)^{\frac{\la}{2}}\int_0^{t\land \si}\la^2\left(||u||^2+\eta\right)^{\frac{\la-4}{2}}\left|\left|((G(u),u))\right|\right|_{L_2(\mathcal{H};\mathbb{R})}^2ds\right)^{\frac{1}{2}} \\
    & \leq \frac{1}{2}\mathbb{E}\sup_{s\in [0,t\land \si]}\left(||u||^2+\eta\right)^{\frac{\la}{2}}+C\mathbb{E}\left(\int_0^{t\land \si}\la^2\left(||u||^2+\eta\right)^{\frac{\la-4}{2}}\al^2 ||u||^4 ds\right). 
    \end{aligned}
\end{equation}
Integrating \eqref{G1} from 0 to s, taking the supremum for $s\in [0,t\land \si]$, taking expectation and combining with \eqref{G2}-\eqref{G4}, we deduce that
\begin{equation}\label{G5}
    \begin{aligned}
    \mathbb{E} & \sup_{s\in [0,t\land \si]} \left(||u||^2+\eta\right)^{\frac{\la}{2}}  \leq \mathbb{E}\left(||u_0||^2+\eta\right)^{\frac{\la}{2}}-\frac{\la(\la-2)}{2}\be^2 \eta^{\frac{\la}{2}}t \\
    &+\mathbb{E}\int_0^{t\land \si}\la(||u||^2+\eta)^{\frac{\la-2}{2}}||u||^2\left\{C_1||u||^r+\frac{\al^2}{2}+\frac{\la-2}{2}\be^2+C_2\la \al^2 \right\} ds.
    \end{aligned}
\end{equation}
Let
$$\xi=\frac{2\be^2-\al^2}{4},\,\,\,\,\,\,\,\,\,\la=\frac{2\be^2-\al^2}{2\be^2+4C_2\al^2},$$
and define
$$\si_\xi':=\inf \left\{t>0: ||u(t)||\geq \left(\frac{\xi}{C_1}\right)^{\frac{1}{r}}\right\},$$
where $C_1$ and $C_2$ are the constants that occur on the right hand side of \eqref{G5}. From Lemma \ref{GP} we can  replace $\si$ in \eqref{G5} by $\si_\xi'$, note that at this time the integrand in the integral on the right hand side of \eqref{G5} is non-positive, thus we have
\begin{equation}\label{G6}
    \mathbb{E}\sup_{s\in [0,t\land \si_\xi']}\left(||u||^2+\eta\right)^{\frac{\la}{2}}\leq \mathbb{E}\left(||u_0||^2+\eta\right)^{\frac{\la}{2}}-\frac{\la(\la-2)}{2}\be^2\eta^{\frac{\la}{2}} t.
\end{equation}
Letting $\eta \rightarrow 0$ in \eqref{G6}, we obtain
\begin{equation}\label{G9}
    \mathbb{E}\sup_{s\in [t\land \si_\xi']}||u(s)||^\la \leq \mathbb{E}||u_0||^\la \leq (\mathbb{E}||u_0||)^\la,
\end{equation}
the second inequality holds due to $\la \in (0,1)$. Using Lemma \ref{GP}, Chybeshev's inequality and \eqref{G9}, we finally conclude that
\begin{equation}\label{G7}
\begin{aligned}
\mathbb{P}(\tau^M=\infty)
& \geq \mathbb{P}(\sigma_\xi'=\infty)  = \mathbb{P}\left(\bigcap_{n=1}^\infty\{\sigma_\xi' > n\}\right)  = \lim_{n \rightarrow \infty} \mathbb{P}\left(\{\sigma_\xi' > n\}\right) \\
& = \lim_{n \rightarrow \infty} \mathbb{P}\left(||u(\sigma_\xi' \land n)||^r <\frac{\xi}{C_1}\right)  \geq 1- \lim_{n \rightarrow \infty} \frac{\mathbb{E}||u(\sigma_\xi' \land n )||^\lambda}{\left(\frac{\xi  }{C_1} \right)^{\frac{\lambda}{r}}} \\
& \geq 1- \left(\frac{C_1^{\frac{1}{r}} \mathbb{E}||u_0||}{\xi
^{\frac{1}{r}}}\right)^\lambda = 1-\left(\frac{C_1 ^{\frac{1}{r}} \mathbb{E}||u_0||}{(\frac{2\beta^2-\alpha^2}{4} )^\frac{1}{r}}\right)^{\frac{2\beta^2-\alpha^2}{2\beta^2+4C_2 \alpha^2}}.
\end{aligned}
\end{equation}
Therefore, for any $0< \epsilon <1$, taking the constant $\delta=\delta(\epsilon)>0$ satisfying
\begin{equation}\label{G8}
   \left(\frac{C_1 ^{\frac{1}{r}} \delta}{(\frac{2\beta^2-\alpha^2}{4} )^\frac{1}{r}}\right)^{\frac{2\beta^2-\alpha^2}{2\beta^2+4C_2 \alpha^2}} = \epsilon,
\end{equation}
from \eqref{G7} we obtain that for any initial data $u_0$ satisfying $\mathbb{E}||u_0|| \leq \delta$, it holds that
\begin{equation}
    \mathbb{P}(\tau^M=\infty)  \geq 1-\epsilon,
\end{equation}
which completes the proof of Theorem \ref{GPT}. \qed


\section{Well-posedness of the stochastic Navier-Stokes equations}\label{NS}

In this section we use Theorems \ref{MPS} and \ref{GPT} to study the following initial boundary value problem for the stochastic Navier-Stokes equations in a bounded domain $D\subset \mathbb{R}^d$ (d=2,3):

\begin{equation}\label{SNS}
    \left\{
    \begin{aligned}
       du+\left[(u\cdot \nabla) u + \nabla p\right] dt &= \nu \Delta u dt  +G(u)dW,\\
       \nabla \cdot u &=0,\\
       u(0) &=u_0,\\
       u|_{\partial D} &=0,
    \end{aligned}
\right.
\end{equation}
where $u=(u_1,\cdots,u_d)^T$, $p$ and $\nu$ denote the velocity, pressure and the coefficient of kinematic viscosity respectively.

Define
\begin{equation}
    H:=\{u \in L^2(D)^d : \nabla \cdot u =0 , \,\,\, u \cdot \mathop{n}\limits^{\rightarrow}=0 \,\,\, \mathrm{on} \,\,\, \partial D \},
\end{equation}
where $\mathop{n}\limits^{\rightarrow}$ is the outer unit normal to $\partial D$, and we take the $L^2$ inner product and norm on $H$ as
\begin{equation}
    (u,v):=\int_D u\cdot v dx, \,\,\,\,\,\,\,\,\, |u|:=\sqrt{(u,u)}.
\end{equation}
Define also
\begin{equation}
    V:=\{u\in H^1_0(D)^d : \nabla \cdot u =0\},
\end{equation}
the inner product and norm on which are chosen as
\begin{equation}\label{VN}
    ((u,v)):=\int_D \nabla u \cdot \nabla v dx, \,\,\,\,\,\,\,\,\, ||u||:=\sqrt{((u,u))}.
\end{equation}
Using the Dirichlet boundary condition $\eqref{SNS}_4$, we know that the Poincaré inequality
\begin{equation}
    |u| \leq C ||u||, \,\,\,\,\,\,\,\,\, \forall u \in V,
\end{equation}
holds, from which one can verify that $||\cdot||$ indeed a norm and is equivalent to the $H^1$ norm.

Denote the Leray projection by $P_H$, which is the orthogonal projection of $L^2(D)^d$ onto $H$. By acting $P_H$ on \eqref{SNS}, and noting the hypotheses on $G$, we can rewrite the problem \eqref{SNS} as follows:
\begin{equation}\label{SNS2}
    \left\{
    \begin{aligned}
       du+[\nu Au +B(u,u)]dt &=G(u)dW,\\
       u(0)&=u_0,
    \end{aligned}
\right.
\end{equation}
where $B(u,v)=P_H(u\cdot \nabla)v$, $Au=-P_H \Delta u$.

The operator $\nu A$ satisfies the hypotheses in Section \ref{HA} is standard, the detail can be found in \cite{C-F} or \cite{Temam}. In order to use Theorems \ref{MPS} and \ref{GPT}, we need to show that $B(\cdot , \cdot)$ satisfies the hypotheses \eqref{B1}-\eqref{HGP}.

\begin{lemma}\label{ABFN}
The bilinear term $B(u,v)=P_H (u\cdot \nabla)v$ is continuous from $V\times D(A)$ to $V'$ and is continuous from $D(A)\times D(A)$ to H. Further, we have
\begin{itemize}
    \item [(1)] $\langle B(u, v) , v \rangle =0 , \,\,\,\,\,\,\,\,\, \forall \,\,u\in V,\,\,v \in D(A),$
    \item [(2)] $|\langle B(u,v) , w \rangle| \leq C ||u||\,|Av|\,||w||,\,\,\,\,\,\,\,\,\, \forall \,\,u,w \in V ,\,\, v\in D(A),$
    \item [(3)] $|(B(u,v),w)|\leq C ||u||^{\frac{1}{2}} \, |Au|^{\frac{1}{2}} \, ||v||^{\frac{1}{2}} \, |Av|^{\frac{1}{2}} \, |w| , \,\,\,\,\,\,\,\,\, \forall u,v \in D(A),\,\,w \in H.$
\end{itemize}
\end{lemma}
\begin{proof}
We only show the case $d=3$, the proof for the case $d=2$ is similar. Fix any divergence free fields $u,v,w\in ( C_0^\infty (D) )^3$. The statement $(1)$ is obvious.

For statement $(2)$, using Hölder's inequality and Ladyzhenskaya's inequality, we obtain
\begin{equation}\label{NSB1}
\begin{aligned}
    |\langle B(u,v),w \rangle|
    & \leq \int_D |(u\cdot \nabla)v \cdot w| dx \\
    & \leq ||u||_{L^4} \, |\nabla v| \, ||w||_{L^4} \\
    & \leq C |u|^{\frac{1}{4}} \, ||u||^{\frac{3}{4}} \, |\nabla v| \,  |w|^{\frac{1}{4}} \, ||w||^{\frac{3}{4}} \\
    & \leq C ||u|| \, |Av| \, ||w||.
    \end{aligned}
\end{equation}

 For statement $(3)$, using Hölder's inequality, Sobolev inequality and the interpolation inequality, we have
\begin{equation}\label{NSB2}
\begin{aligned}
    |( B(u,v),w )|
    & \leq \int_D |(u\cdot \nabla)v \cdot w \cdot 1| dx \\
    & \leq C ||u||_{L^{12}}  ||\nabla v||_{L^3}   |w| \\
    & \leq C ||u||_{H^{\frac{3}{2}}}||\nabla v||_{H^{\frac{1}{2}}}|w| \\
    & \leq C ||u||_{H^1}^{\frac{1}{2}}||u||_{H^2}^{\frac{1}{2}} |\nabla v|^{\frac{1}{2}} ||\nabla v||_{H^1}^{\frac{1}{2}}|w|\\
    & \leq C ||u||^{\frac{1}{2}}|Au|^{\frac{1}{2}}||v||^{\frac{1}{2}}|Av|^{\frac{1}{2}}|w|.
    \end{aligned}
\end{equation}
From \eqref{NSB1} and \eqref{NSB2}, we conclude statements $(2)$ and $(3)$ by density argument.

\end{proof}



\begin{lemma}\label{ABFN2}
The bilinear term $B(u,v)$ is also continuous from $V \times D(A)$ to $H$, and when $u\in V$, $v\in D(A)$, $w\in H$ we have
\begin{equation}\label{EB}
|(B(u,v),w)|\leq C \begin{cases}
  |u|^{\frac{1}{2}}||u||^{\frac{1}{2}}||v||^{\frac{1}{2}}|Av|^{\frac{1}{2}}|w|,\,\,\,\,\,\,\,\,\,\,\,\,\,d=2, \\
  ||u||\,||v||^{\frac{1}{2}}|Av|^{\frac{1}{2}}|w|,\,\,\,\,\,\,\,\,\,\,\,\,\,\,\,\,\,\,\,\,\,\,\,\,\,\,\, d=3.
  \end{cases}
  \end{equation}
In particular, for $u \in D(A)$,
\begin{equation}\label{EB1}
    |(B(u,u), Au)|\leq C ||u||^{\frac{3}{2}}|Au|^{\frac{3}{2}},
\end{equation}
holds for both $d=2$ and $d=3$.
\end{lemma}
\begin{proof}
We only show \eqref{EB} for d=3. Fix any divergence free fields $u,v,w\in ( C_0^\infty (D) )^3$, using Hölder's inequality, Sobolev inequality and the interpolation inequality, we have
\begin{equation}
    \begin{aligned}
    |(B(u,v),w)| & \leq ||u||_{L^6}||\nabla v||_{L^3} |w|      \\
    & \leq C ||u||_{H^1}||\nabla v||_{H^{\frac{1}{2}}} |w|         \\
    & \leq C ||u||\,||\nabla v||_{L^2}^{\frac{1}{2}}||\nabla v||_{H^1}^{\frac{1}{2}} |w|         \\
    & \leq C ||u||\,||v||^{\frac{1}{2}}|Av|^{\frac{1}{2}}|w|,
    \end{aligned}
\end{equation}
from which we obtain \eqref{EB} by density argument, and \eqref{EB1} follows immediately.

\end{proof}

Lemmas \ref{ABFN} and \ref{ABFN2} indicate that the hypotheses \eqref{B1}-\eqref{HGP} hold for the bilinear term of the Navier-Stokes equations, therefore, we obtain Theorem \ref{GPSSNS} by applying Theorems \ref{MPS} and \ref{GPT} directly to the problem \eqref{SNS2}.


\section{Global pathwise solutions for stochastic Navier-Stokes equations defined on torus}\label{NSM}
In this section, we study the problem \eqref{SNS} defined on the torus $\mathbb{T}^d$ with $d=2,3$. Denote by $X^m$ the divergence free subspace of the usual Sobolev space $H^m(\mathbb{T}^d)$, that is, $X^m=\{u \in H^m : \nabla \cdot u = 0 \}$. Assume that the initial data $u_0$ is a $X^m$-valued, $\mathcal{F}_0$-measurable random variable with $m>\frac{d}{2}+1$, and the noise term $G$ satisfies

  \begin{equation}\label{GPSm-H1}
        G\in Bnd_{u,loc}(H^{m+1},L_2(\mathcal{H},X^{m+1});L^\infty) \cap Bnd_{u,loc}(H^{m+3},L_2(\mathcal{H},X^{m+3});L^\infty),
    \end{equation}

   \begin{equation}\label{GPSm-H2}
        G\in Lip_{u,loc}(H^m,L_2(\mathcal{H},X^m);L^\infty) \cap Lip_{u,loc}(H^{m-1},L_2(\mathcal{H},X^{m-1});L^\infty),
    \end{equation}

 \begin{equation}\label{GPSm-H3}
    ||(G(u),u)_{H^m}||_{L_2{(\mathcal{H},\mathbb{R})}} ^2 \geq \beta ^2 ||u||_{H^m}^4,
\end{equation}
and
\begin{equation}\label{GPSm-H4}
    ||G(u)||_{L_2{(\mathcal{H},H^m)}} ^2 \leq \alpha^2 ||u||_{H^m}^2,
\end{equation}
for some constants $\al$ and $\be$ satisfying \eqref{G-H3}.

\subsection{Existence of maximal $H^m$ pathwise solutions}\label{MPSM}
As in Section \ref{LMPS}, we first consider the initial data which is almost surely bounded in the $W^{1,\infty}$-norm, specifically,
\begin{equation}\label{BI}
     ||u_0||_{W^{1,\infty}}\leq M,\,\,\,\,\,\,\,\,\,\,\mathbb{P}-\mathrm{a.s.},
\end{equation}
for some positive constant $M$. Inspired by \cite{TZ}, we consider another cut-off system:
\begin{equation}\label{CS2}
\left\{
\begin{aligned}
du+[\nu Au + \theta(||u||_{W^{1,\infty}})B(u,u)]dt &=\theta(||u||_{W^{1,\infty}}) G(u) dW,\\
u(0) &=u_0,
\end{aligned}
\right.
\end{equation}
where $\theta$ is defined in \eqref{CF} with $\ka>M$,  and the corresponding Galerkin system:
\begin{equation}\label{GS2}
\left\{
\begin{aligned}
 du^n+[\nu Au^n + \theta(||u^n||_{W^{1,\infty}})B^n(u^n,u^n)]dt &=\theta(||u^n||_{W^{1,\infty}})G^n(u^n) dW,\\
 u^n(0) &=P_nu_0,
\end{aligned}
\right.
\end{equation}
where $B^n(u,u)=P_nB(u,u)$ and $G^n(u)=P_nG(u)$. Denote the solution of \eqref{GS2} by $u^n$, the existence and uniqueness of which is standard, and also $u^n \in C \left([0,\infty); P_n X^{m+3} \right)$ almost surely, we refer to \cite{Flandoli, B-F-H} for further detail. Let's establish the following uniform estimates.

\begin{proposition}
     Let $\mathcal{S}=(\Om, \mathcal{F}, (\mathcal{F}_t)_{t \geq 0}, \mathbb{P}, W)$ be a fixed stochastic basis, assume that $m>\frac{d}{2}+1$, $u_0 \in L^p(\Omega,\mathcal{F}_0,\mathbb{P};X^{m+3})$ for some $p \geq 2$, satisfying \eqref{BI}, and G(u) satisfying the hypotheses \eqref{GPSm-H1}-\eqref{GPSm-H2}, then for any $T>0$, there is a constant $C=C(p,\ka,T,||u_0||)>0$, such that the approximate solutions $\{u^n\}_{n\in \mathbb{N}}$ satisfying
  \begin{equation}\label{UEM1}
    \sup_{n \in \mathbb{N}} \mathbb{E} \left( \sup_{t\in [0,T]}||u^n||^p_{H^{m+3}} + \int_0^T \nu ||u^n||^{p-2}_{H^{m+3}}||\nabla u||_{H^{m+3}}^2 ds \right) \leq C,
  \end{equation}
  and
  \begin{equation}\label{UEM2}
     \sup_{n \in \mathbb{N}} \mathbb{E}||u^n||^p_{W^{\alpha,p}([0,T];H^{m+2})} \leq C,
  \end{equation}
  where $\al \in (0,\frac{1}{2})$.
\end{proposition}

\begin{proof}
Let $m_1=m+3$, the Itô formula for $||u^n||^p_{H^{m_1}}$ yields
\begin{equation}\label{uem1}
\begin{aligned}
   d||u^n||^p_{H^{m_1}}+ & \nu p ||u^n||_{H^{m_1}}^{p-2} (Au^n,u^n)_{H^{m_1}} dt  \\
 = &  -p ||u^n||^{p-2}_{H^{m_1}} \theta(||u^n||_{W^{1,\infty}}) (B^n(u^n,u^n),u^n)_{H^{m_1}} dt \\
&  + \frac{p(p-2)}{2} ||u^n||^{p-4}_{H^{m_1}} \theta^2(||u^n||_{W^{1,\infty}}) ||(G^n(u^n),u^n)_{H^{m_1}}||_{L_2(\mathcal{H},\mathbb{R})}^2 dt \\
&  +\frac{p}{2} ||u^n||^{p-2}_{H^{m_1}} \theta^2(||u^n||_{W^{1,\infty}}) ||G^n(u^n)||^2_{L_2(\mathcal{H},H^{m_1})} dt \\
&  + p||u^n||^{p-2}_{H^{m_1}} \theta(||u^n||_{W^{1,\infty}}) (G^n(u^n),u^n)_{H^{m_1}} dW \\
= &   (I_1+I_2+I_3)dt +I_4 dW.
\end{aligned}
\end{equation}
Using \eqref{EFB2}, we estimate the first term as the following:
\begin{equation}\label{uem2}
    \begin{aligned}
    |I_1| & \leq C \theta(||u^n||_{W^{1,\infty}}) ||u^n||^p_{H^{m_1}} ||u^n||_{W^{1,\infty}} \\
    & \leq C ||u^n||_{H^{m_1}}^p.
    \end{aligned}
\end{equation}
The hypothesis \eqref{GPSm-H1} implies
\begin{equation}\label{uem3}
    \begin{aligned}
    |I_2| + |I_3| & \leq C \theta^2(||u^n||_{W^{1,\infty}}) ||u^n||^{p-2}_{H^{m_1}} ||G(u^n)||_{L_2(\mathcal{H},H^{m_1})}^2 \\
    &  \leq C \theta^2(||u^n||_{W^{1,\infty}}) ||u^n||^{p-2}_{H^{m_1}} \beta^2(||u^n||_{L^\infty})\left(1+||u^n||_{H^{m_1}}^2\right) \\
    &  \leq C \beta^2(2\kappa) \left(1+||u^n||^p_{H^{m_1}}\right) \\
    &  \leq C \left(1+||u^n||^p_{H^{m_1}}\right).
    \end{aligned}
\end{equation}
To treat the stochastic term $I_4$, for $\xi>0$, introducing the stopping time
\begin{equation}
    \Tilde{\si}_\xi=\inf \{t > 0 : \sup_{s \in [0,t]}||u^n||_{H^{m_1}} \geq \xi\}.
\end{equation}
For any pair of stopping times $\tau_a$ and $\tau_b$, satisfying $0 \leq \tau_a \leq \tau_b \leq \Tilde{\si}_\xi \land T$, by using Lemma \ref{BDG} and the hypothesis \eqref{GPSm-H1}, we have
\begin{equation}\label{uem5}
    \begin{aligned}
   & \mathbb{E}\sup_{s\in[\tau_a,\tau_b]} \left|\int_{\tau_a}^s I_4 dW \right| \\& \leq C \mathbb{E} \left(\int_{\tau_a}^{\tau_b} p^2 ||u^n||^{2(p-2)} \theta^2(||u^n||_{W^{1,\infty}}) ||(G(u^n),u^n)_{H^{m_1}}||^2_{L_2(\mathcal{H},\mathbb{R})} ds \right)^{\frac{1}{2}}  \\
    & \leq C \mathbb{E} \left(\sup_{s \in [\tau_a,\tau_b]}||u^n||^{\frac{p}{2}}_{H^{m_1}}\left(\int_{\tau_a}^{\tau_b}||u^n||^{p-2}_{H^{m_1}}\theta^2(||u^n||_{W^{1,\infty}}) \beta^2(||u^n||_{L^\infty}) (1+||u^n||_{H^{m_1}})^2 ds\right)^{\frac{1}{2}}\right) \\
    & \leq \frac{1}{2} \mathbb{E}\left(\sup_{s\in [\tau_a,\tau_b]}||u^n||^p_{H^{m_1}}\right)+C\mathbb{E}\left(\int_{\tau_a}^{\tau_b}||u^n||^{p-2}_{H^{m_1}}\theta^2(||u^n||_{W^{1,\infty}}) \beta^2(||u^n||_{L^\infty}) (1+||u^n||_{H^{m_1}})^2 ds\right) \\
    & \leq \frac{1}{2} \mathbb{E}\left(\sup_{s\in [\tau_a,\tau_b]}||u^n||^p_{H^{m_1}}\right)+C\mathbb{E}\left(\int_{\tau_a}^{\tau_b} (1+||u^n||_{H^{m_1}}^p ) ds\right).
    \end{aligned}
\end{equation}
Thus integrating \eqref{uem1} from $\tau_a$ to $t$, and taking the supermum for $t\in [\tau_a,\tau_b]$, taking expectation, and using \eqref{uem2}-\eqref{uem5}, we conclude that
\begin{equation}\label{GF}
    \begin{aligned}
    & \mathbb{E} \left(\sup_{s\in[\tau_a,\tau_b]}||u^n||^p_{H^{m_1}} + \nu \int_{\tau_a}^{\tau_b} ||u^n||^{p-2}_{H^{m_1}}||\nabla u||_{H^{m_1}}^2 ds \right) \\
    & \leq C\mathbb{E}\left(||u^n(\tau_a)||^p_{H^{m_1}}+ \int_{\tau_a}^{\tau_b} (1+||u^n||_{H^{m_1}}^p ) \, ds. \right).
    \end{aligned}
\end{equation}
From \eqref{GF}, by applying Lemma \ref{SGL}, one may obtain
\begin{equation}
    \begin{aligned}
    & \mathbb{E}\left(\sup_{s\in[0,\Tilde{\si}_\xi \land T]}||u^n||^p_{H^{m_1}} + \nu \int_0^{\Tilde{\si}_\xi \land T} ||u^n||^{p-2}_{H^{m_1}}||\nabla u||_{H^{m_1}}^2 ds \right) \\
    & \leq C \mathbb{E}\left(||u^n(0)||^p_{H^{m_1}}+ \int_0^{\Tilde{\si}_\xi \land T} ds\right) \\
    & \leq C \mathbb{E}\left(||u(0)||^p_{H^{m_1}}+1\right) \\
    & \leq C.
    \end{aligned}
\end{equation}
It is clear that $\Tilde{\si}_\xi \rightarrow \infty$ as $\xi \rightarrow \infty$ since $u^n \in C([0,\infty) ; P_n X^{m_1})$, thus \eqref{UEM1} holds by letting $\xi \rightarrow \infty$.
From \eqref{UEM1} and the hypothesis \eqref{GPSm-H1}, we deduce
\begin{equation}
   \mathbb{E} \left|\left|\int_0^t \theta (||u^n||_{W^{1,\infty}}) G^n(u^n)dW \right|\right|^p_{W^{\al,p}([0,T],H^{m_1-1})} \leq C,
\end{equation}
from which and the estimate for the bilinear term in Lemma \ref{EFB}, we finally conclude \eqref{UEM2}.
\end{proof}

Performing the similar arguments as in Section \ref{LMPS} (see also \cite{TZ} for the treatment of the initial data), we can obtain that for any $X^m$-valued, $\mathcal{F}_0$-measurable random variable $u_0$, there exists a unique maximal pathwise solution, continuously evolving in $X^m$, to the problem \eqref{SNS} with the initial data $u_0$.

\subsection{Existence of global $H^m$ pathwise solutions}
In this subsection, we will complete the proof of Theorem \ref{GPSmT}.

\noindent \textbf{Proof of Theorem \ref{GPSmT}:}
The proof is similar to Theorem \ref{GPT}. Let $(u,\tau^M)$ be the maximal pathwise solution established in Subsection \ref{MPSM}, we will show that the probability of the set $\{ \tau^M = \infty \}$ is positive and uniformly about $\nu$.

Fix any $\eta>0$ and $\la \in (0,1) $, the Itô formula for $\left(||u||_{H^m}^2+\eta\right)^{\frac{\la}{2}}$ yields
\begin{equation}
    \begin{aligned}
    & d\left(||u||_{H^m}^2 +\eta\right)^{\frac{\la}{2}} \\
    &  = \la \left(||u||_{H^m}^2+\eta\right)^\frac{\la -2}{2} \left\{\left(-\nu Au -B(u,u),u\right)_{H^m}+\frac{1}{2} ||G(u)||_{L_2(\mathcal{H},H^m)}^2 \right\} dt \\
    &\,\,\,\,\,\,  + \frac{\la (\la -2)}{2} \left(||u||_{H^m}^2+\eta\right)^{\frac{\la -4}{2}} ||(G(u),u)_{H^m}||_{L_2(\mathcal{H}, \mathbb{R})} ^2 dt \\
    &\,\,\,\,\,\, + \la \left(||u||_{H^m}^2+\eta\right)^{\frac{\la -2}{2}} (G(u),u)_{H^m} dW  \\
    & : = (I_1+I_2) dt + I_3 dW.
    \end{aligned}
\end{equation}

Using \eqref{EFB2} and the embedding $H^m \subset W^{1, \infty}$, we have
\begin{equation}
    |(B(u,u),u)_{H^m}|\leq C||u||^3_{H^m},
\end{equation}
combining the hypothesis \eqref{GPSm-H4}, we estimate the first term as the following:
\begin{equation}
    |I_1| \leq \la (||u||_{H^m}^2+\eta)^{\frac{\la -2}{2}}||u||_{H^m}^2\left(C||u||_{H^m}+\frac{1}{2} \al^2\right).
\end{equation}

To treat the terms $I_2$ and $I_3$, we just need to change the $V$ norm in \eqref{G2}-\eqref{G3} to the $H^m$ norm. Consequently, we establish the estimate which is similar to \eqref{G4}, more precisely, for any stopping time $\si$ such that $\si<\tau^M$ on the set $\{\tau^M<\infty \}$, we obtain
\begin{equation}\label{GPSm1}
    \begin{aligned}
 \mathbb{E} & \left(||u(t\land \sigma||_{H^m}^2+\eta\right)^{\frac{\lambda}{2}}
 \leq \mathbb{E}\left(||u_0||_{H^m}^2+\eta\right)^{\frac{\lambda}{2}} -C \frac{\lambda (\lambda-2)}{2}\beta^2\eta^{\frac{\lambda}{2}}t \\
& +\mathbb{E}\int_0^{t \land \sigma} \lambda \left(||u||_{H^m}^2+\eta\right)^{\frac{\lambda-2}{2}}||u||_{H^m}^2 \left\{C_3 ||u||_{H^m}+\frac{\alpha^2}{2}+\frac{\lambda-2}{2}\beta^2+C_4 \lambda \alpha^2 \right\} ds.
\end{aligned}
\end{equation}

Similarly, let 
\begin{equation}
    \xi=\frac{2\beta^2-\alpha^2}{4},\,\,\,\,\,\,\,\,\,\lambda=\frac{2\beta^2-\alpha^2}{2\beta^2+4C_4 \alpha^2},
\end{equation}
and for $\xi>0$, define the stopping time
\begin{equation}
    \Bar{\si}_\xi := \inf \{t>0 :\sup_{s\in [0,t]}||u||_{H^m} \geq \frac{\xi}{C_3} \},
\end{equation}
where $C_3$ and $C_4$ are the constants that occur on the right hand side of \eqref{GPSm1}. It is clear that
\begin{equation}
    \Bar{\si}_\ga \leq \si_{\ga'}' < \tau^M,
\end{equation}
for any $\ga>0$ with $ \ga' = C_3 \left( \frac{\ga}{C_1} \right) ^{\frac{1}{r}}$ and almost surely $\om \in \{\tau^M < \infty \}$.
Therefore, we can  replace $\si$ in \eqref{GPSm1} by $\Bar{\si}_\xi$, note that at this time the third term on the right hand side of \eqref{GPSm1} is non-positive, we deduce
\begin{equation}\label{GPSm3}
    \mathbb{E} \left(||u(t\land \sigma||_{H^m}^2+\eta\right)^{\frac{\lambda}{2}}\leq \mathbb{E}\left(||u_0||_{H^m}^2+\eta\right)^{\frac{\lambda}{2}} -C \frac{\lambda (\lambda-2)}{2}\beta^2\eta^{\frac{\lambda}{2}}t.
\end{equation}
By passing $\eta \rightarrow 0$ in \eqref{GPSm3}, we have
\begin{equation}
    \mathbb{E}||u(t\land \sigma)||_{H^m}^\lambda \leq \mathbb{E}||u_0||_{H^m}^\lambda \leq (\mathbb{E}||u_0||_{H^m})^\lambda.
\end{equation}
Therefore, following the same procedure as in \eqref{G7}, we finally conclude that
\begin{equation}\label{GPSm2}
\begin{aligned}
\mathbb{P}(\tau^M=\infty)
\geq  1- \left(\frac{4C_3\mathbb{E}||u_0||_{H^m}}{2\beta^2-\alpha^2}\right)^{\frac{2\beta^2-\alpha^2}{2\beta^2+4C_4 \alpha^2}}.
\end{aligned}
\end{equation}
The proof of Theorem \ref{GPSmT} is complete.
\qed

\appendix

\section{Some estimates}

\numberwithin{equation}{section}

\setcounter{equation}{0}


In this section, we give some estimates used in this paper. We give the proof of Lemma \ref{GP}, which plays a crucial role in the proof of the result for the global solutions. In addition, some properties of the bilinear term $B(\cdot , \cdot)$ for the Navier-Stokes equations have been presented.

\subsection{Proof of Lemma \ref{GP}}\label{PKL}

 Let $\{ \tau_n \}_{n\in \mathbb{N}}$ announces any finite time blow up. For any $t \geq 0 $ and almost surely $\omega \in  \Omega$, the Itô formula for $||u||^2$ yields
\begin{equation}\label{1}
    \begin{aligned}
    ||u(t\land \tau_n \land \si_\xi)||^2
    & =||u_0||^2-2\int_0^{t\land \tau_n \land \si_\xi }|Au|^2\, ds - 2\int_0^{t\land \tau_n \land \si_\xi } (B(u),Au) \, ds    \\
    & +\int_0^{t \land \tau_n \land \si_\xi} ||G(u)||_{L_2(\mathcal{H};V)}^2 \, ds + 2\int_0^{t \land \tau_n \land \si_\xi}((G(u),u)) \, dW.
    \end{aligned}
\end{equation}
Using the assumptions \eqref{HGP} and \eqref{G-H2}, we obtain
\begin{equation}\label{2}
    \begin{aligned}
    & 2 \int_0^{t\land \tau_n \land \si_\xi}|(B(u),Au)| \, ds +\int_0^{t \land \tau_n \land \si_\xi} ||G(u)||_{L_2(\mathcal{H};V)}^2 \, ds   \\
    & \leq \int_0^{t\land \tau_n \land \si_\xi}|Au|^2 \, ds +C\int_0^{t \land \tau_n \land \si_\xi}||u||^{2+r} \, ds +\int_0^{t \land \tau_n \land \si_\xi} ||G(u)||_{L_2(\mathcal{H};V)}^2 \, ds   \\
    &\leq \int_0^{t\land \tau_n \land \si_\xi}|Au|^2 \, ds+ C\xi^{2+r}t+C\al^2\xi^2t,
    \end{aligned}
\end{equation}
where $r=\frac{2(a+b-2)}{2-b}>0$. From \eqref{1} and \eqref{2} we have
\begin{equation}\label{3}
    \begin{aligned}
    \left|\int_0^{t \land \tau_n \land \si_\xi}((G(u),u)) \, dW\right|  & \geq \frac{1}{2}||u(t \land \tau_n \land \si_\xi)||^2 - \frac{1}{2}||u_0||^2 \\
    & +\frac{1}{2}\int_0^{t \land \tau_n \land \si_\xi}|Au|^2 \, ds -C\xi^{2+r}t-C\al^2\xi^2t .
    \end{aligned}
\end{equation}

For any $k,n \in \mathbb{N}^+$, define
$$A_k^n :=\{\tau_n\leq \si_\xi\} \cap \{\tau_n \leq k\}.$$
Thus for any fixed $k$ and almost surely $\omega \in A_k^n$, from the hypothesis \eqref{G-H2} there holds that
\begin{equation}
    \int_0^{\tau_n} ||((G(u),u))||_{L_2(\mathcal{H};\mathbb{R})}^2 \, ds \leq \al^2\xi^4k < n,
\end{equation}
provided $n$ large enough, which implies
\begin{equation}\label{6}
    \mathbb{P}\left(A_k^n \cap \left\{\int_0^{\tau_n} ||((G(u),u))||_{L_2(\mathcal{H};\mathbb{R})}^2 \, ds \geq n\right\}\right) = 0,
\end{equation}
for $n$ large enough.

On the other hand, for any fixed $k$ and almost surely $\omega \in A_k^n$, from \eqref{FBU} we know
\begin{equation}\label{4}
    \begin{aligned}
    \int_0^{k \land \tau_n \land \si_\xi}|Au|^2 \, ds & = \int_0^{\tau_n}|Au|^2 \, ds \\
    & = n - \sup_{t\in [0,\tau_n]}||u(t)||^2  \\
    & \geq n - \xi^2,
    \end{aligned}
\end{equation}
from which and \eqref{3} we infer that
\begin{equation}
    \left|\int_0^{k \land \tau_n \land \si_\xi}((G(u),u)) \, dW\right| \geq \frac{n}{2}-\xi^2-C(\xi,k)\geq \frac{n}{3} ,
\end{equation}
provided $n$ large enough. Thus by an inequality from \cite{RS} we have
\begin{equation}\label{5}
    \begin{aligned}
    \mathbb{P}(A_k^n) & =\mathbb{P}\left(A_k^n \cap \left\{\left|\int_0^{\tau_n}((G(u),u))dW\right| \geq \frac{n}{3}\right\}\right)  \\
    & \leq \frac{9}{n}+ \mathbb{P}\left(A_k^n \cap \left\{\int_0^{\tau_n} ||((G(u),u))||_{L_2(\mathcal{H};\mathbb{R})}^2 ds \geq n\right\}\right).
    \end{aligned}
\end{equation}
For any fixed $k$ and large enough $n$, from \eqref{6} and \eqref{5} we deduce
\begin{equation}\label{7}
    \mathbb{P}(A_k^n)\leq \frac{9}{n}.
\end{equation}

For any $k\in \mathbb{N}^+$, let
$$A_k=\{\tau^M \leq \si_\xi\} \cap \{\tau^M \leq k\},$$
thus for any $n\in \mathbb{N}^+$, it is obvious that
$$A_k \subset A_k^n,$$
since $\tau_n \leq \tau^M$. Letting $n\rightarrow \infty$ in \eqref{7}, we find
\begin{equation}\label{8}
    \mathbb{P}(A_k)=0.
\end{equation}
Note
$$\{\tau^M \leq \si_\xi\} \cap \{\tau^M < \infty \} = \bigcup_{k=1}^\infty A_k, $$
from which and \eqref{8} we conclude that
$$\mathbb{P}(\{\tau^M \leq \si_\xi\} \cap \{\tau^M < \infty\})=0. $$
The proof is complete.
\qed
\\

\subsection{Estimates on the bilinear term for the Navier-Stokes equations}

\begin{lemma}\label{EFB}
There exists a constant C depends on $m$ and $d$ such that:

$\mathrm{(1)}$ For $u \in X^m \cap L^ \infty $, $v \in X^{m+1} \cap W^{1,\infty} $,
\begin{equation}\label{EFB1}
    ||B(u,v)||_{H^m}\leq C (||u||_{L^\infty}||v||_{H^{m+1}}+||u||_{H^m}||v||_{W^{1,\infty}}).
\end{equation}

$\mathrm{(2)}$ For $u,v \in X^m \cap W^{1,\infty}$,
\begin{equation}\label{EFB2}
      |(B(u,v),v)_{H^m}|\leq C||v||_{H^m}(||u||_{W^{1,\infty}}||v||_{H^m}+||u||_{H^m}||v||_{W^{1,\infty}}).
\end{equation}

$\mathrm{(3)}$ If $m>\frac{d}{2}+1$, then for $u \in X^m$, $v \in X^{m-1}$,
\begin{equation}\label{EFB3}
    |(B(u,v),v)_{H^{m-1}}| \leq C ||u||_{H^m}||v||_{H^{m-1}}^2.
\end{equation}
\end{lemma}
\begin{proof}
Noting $\nabla \cdot u = 0 $, the proof can be easily completed, we refer to  \cite{Majda}.
\end{proof}

    \section{Several results in stochastic analysis and some compact embedding theorems}
 In this section, we collect several lemmas in stochastic analysis and some compact embedding theorems, which are used in the proof in this paper.  Firstly, we introduce an inequality which will be used frequently throughout this paper.
\begin{lemma}\label{BDG} (Burkholder-Davis-Gundy inequality, \cite[Proposition 2.3.8]{BDE})
Assume that $H$ is a separable Hilbert space, and $G$ is a predictable process lying in $L^2(\Omega;L^2_{loc}([0,\infty),L_2(\mathcal{H},H))$, then for any $p \geq 1$, there exists a constant $C=C(p)$ such that
\begin{equation}\label{BDGI}
    \mathbb{E}\left(\sup_{t\in [0,T]}\left|\left|\int_0^t GdW\right|\right|_X^p \right) \leq C \mathbb{E}\left(\int_0^T ||G||_{L_2(\mathcal{H},H)}^2ds\right)^{\frac{p}{2}}.
\end{equation}
\end{lemma}

In order to show fractional Sobolev regularity in time for the stochastic integral, we need the following inequality given by Flandoli-Gatarek in \cite[Lemma 2.1]{FG}.
\begin{lemma}\label{SL}
Let $H$ be a separable Hilbert space. For any given $p\geq 2$, $\alpha \in [0,\frac{1}{2})$, and predictable process
$G\in L^p(\Omega\times [0,T];L_2(\mathcal{H},H))$, we have
\begin{equation}\label{SI}
    \mathbb{E}\left(\left|\left|\int_0^t GdW\right|\right|^p_{W^{\alpha,p}([0,T];H)}\right) \leq C(p,\alpha,T) \mathbb{E}\left(\int_0^T ||G||^p_{L_2(\mathcal{H},H)}ds\right) .
\end{equation}
\end{lemma}

To use the stochastic compactness argument to obtain the martingale solutions, we need the following Prokhorov and Skorokhod theorems. 
\begin{lemma}\label{Prokhorov}
(Prokhorov theorem, \cite[Theorem 5.1, 5.2]{Bil})
Let $X$ be a Polish space and $\mathcal{M}$ be a collection of probability measures on $X$, then $\mathcal{M}$ is tight if and only if it is relatively weakly compact.
\end{lemma}
\begin{lemma}\label{Skorokhod}
(Skorokhod theorem, \cite[Theorem 11.7.2]{Dud})
Let $X$ be a Polish space and $\{ \mu_n \}_{n \in \mathbb{N}} $ be probability measures on $X$ such that $\mu_n$ converges weakly to $\mu$ as $n\to +\infty$. Then on some probability space, there exist $X$-value random variables $\{ u_n \}_{n \in \mathbb{N}} $ such that the law of $u_n$ is $\mu_n$ for each $n$, and $u_n(\omega) \rightarrow u(\omega)$ in $X$ a.s.
\end{lemma}

In order to pass the limit in the stochastic integral, we need the following lemma given in \cite[Lemma 2.1]{De-G}.
\begin{lemma}\label{SCL} 
Let $(\Omega,\mathcal{F},\mathbb{P})$ be a fixed probability space and $H$ be a separable Hilbert space. Consider a sequence of stochastic bases $\mathcal{S}_n=(\Omega, \mathcal{F},\{ \mathcal{F}_t^n \}_{t \geq 0}, \mathbb{P}, W_n)$, that is a sequence so that each $W_n$ is cylindrical Wiener process $(\mathrm{over}\,\,\mathcal{H})$ with respect to $\mathcal{F}_t^n$. Assume that $\{ G_n \}_{n \geq 1}$ are a collection of $H$-valued, $\mathcal{F}_t^n$ predictable processes such that $G_n \in L^2(0,T;L_2(\mathcal{H},H))$ almost surely. Finally consider $\mathcal{S}=(\Omega, \mathcal{F},\{ \mathcal{F}_t \}_{t \geq 0}, \mathbb{P}, W)$ and $G \in L^2(0,T;L_2(\mathcal{H},H))$, which is $\mathcal{F}_t$ predictable. If
$$  W_n\rightarrow W\,\,\,\,\,\, \mathrm{in}\,\,\, C([0,T];\mathcal{H}_0)\,\,\, \mathrm{in} \,\,\,\mathrm{probability},$$
$$ G_n\rightarrow G \,\,\,\,\,\,\mathrm{in} \,\,\,L^2(0,T;L_2(\mathcal{H},H)) \,\,\,\mathrm{in}\,\,\, \mathrm{probability},$$
then
$$    \int_0^t G_n dW_n \rightarrow \int_0^t G dW \,\,\,\,\,\,\mathrm{in}\,\,\, L^2(0,T;H)\,\,\, \mathrm{in}\,\,\, \mathrm{probability}.   $$
\end{lemma}

In order to pass from the martingale solutions to the pathwise solutions, we need a characteristic of convergence in the original probability space, which is stated as below.
\begin{lemma}\label{Gyongy-Krylov}
(Gyongy-Krylov lemma,\cite[Lemma 1.1]{GK})
Let $X$ be a Polish space equipped with the Borel $\sigma$-algebra. A sequence of $X$-valued random variables $\{u_n\}_{n\in \mathbb{N}}$ converges in probability if and only if for every sequence of joint laws of $\{ (u_{n_k},u_{m_k}) \}_{k\in \mathbb{N}}$ there exists a further subsequence which converges weakly to a probability measure $\mu$ such that
$$\mu\{(x,y)\in X\times X : x=y\}=1. $$
\end{lemma}

We also need the following compact embedding results, see Theorems 2.1 and 2.2 in \cite{FG} for statements $(1)$ and $(2)$ respectively.
\begin{lemma}\label{EL}
$\mathrm{(1)}$ Suppose that $X\subset Y\subset Z$ are Banach spaces and $X$, $Z$ reflexive and the embedding $X$ into $Y$ is compact. Then for any $1<p<\infty$ and $0<\alpha<1$, the embedding
\begin{equation}
    L^p([0,T];X)\cap W^{\alpha,p}([0,T];Z)\subset \subset L^p([0,T];Y)
\end{equation}
is compact.

$\mathrm{(2)}$ Suppose that $X\subset Y$ are Banach spaces and the embedding is compact, let $1<p<\infty$ and $0<\alpha \leq 1$ such that $\alpha p >1$ then
\begin{equation}
    W^{\alpha,p}([0,T];X)\subset \subset C([0,T];Y)
\end{equation}
is compact.
\end{lemma}

Finally, we recall a Gronwall lemma for stochastic processes, given in \cite[Lemma 5.3]{G-H-Z}, which is used frequently throughout this paper.
\begin{lemma}\label{SGL}
Fix $T>0$. Assume that $X$,$Y$,$Z$,$R$ : $[0,T)\times \Om \rightarrow \mathbb{R}$ are real-valued, non-negative stochastic processes. Let $\tau < T$ be a stopping time so that
\begin{equation}
    \mathbb{E}\int_0^\tau (RX+Z) ds < \infty.
\end{equation}
Moreover, assume that
\begin{equation}
    \int_0^\tau R ds < M,\,\,\,\,\,\,\,\,\,\,a.s.,
\end{equation}
for some fixed constant $M$. For all stopping times $\tau_a$ and $\tau_b$ satisfying $0 \leq \tau_a \leq \tau_b \leq \tau $, if
\begin{equation}
    \mathbb{E} \left(\sup_{t\in [\tau_a,\tau_b]}X + \int_{\tau_a}^{\tau_b} Y ds \right) \leq C_0  \left( \mathbb{E}(X(\tau_a)+\int_{\tau_a}^{\tau_b}(RX+Z) ds \right),
\end{equation}
where $C_0$ is a constant independent of the choice of $\tau_a$ and $\tau_b$. Then
\begin{equation}
    \mathbb{E} \left(\sup_{t \in [0,\tau] } X + \int_o^\tau Y ds \right) \leq C \mathbb{E} \left( X(0)+\int_0^\tau Z ds \right),
\end{equation}
where $C=C(C_0,T,M)$.
\end{lemma}

\vspace{.1in}
\par{\bf Acknowledgements.} This research was supported by National Natural Science Foundation of China under Grant Nos.12171317, 12331008, 12250710674 and 12161141004, and Shanghai Municipal Education Commission under grant 2021-01-07-00-02-E00087.

\end{document}